\documentclass[11pt, letterpaper]{amsart}
\usepackage{fancyhdr}
\usepackage{txfonts}
\usepackage{comment}
\usepackage{amsmath,amsthm,amssymb,hyperref, mdframed}
\usepackage{mathrsfs} 
\usepackage{tikz}
\usepackage[all]{xy}
\usetikzlibrary{decorations.markings}
\usetikzlibrary{backgrounds,shapes}
\usetikzlibrary{patterns,hobby}
        \usepackage{pgfplots}
        \pgfplotsset{compat=1.6}

\usepackage{subfig}

\usepackage{amsfonts,dsfont,euscript,enumerate,verbatim,calc,mathtools}

\usepackage{color}
\usepackage{hyperref}
\hypersetup{
    colorlinks=true, 
    linktoc=all,     
    linkcolor=blue,  
    citecolor=blue,
    filecolor=blue,
    urlcolor=blue
}

\definecolor{aliceblue}{rgb}{0.9, 0.95, 1.0}
\definecolor{pallido}{RGB}{221,227,227}

\usepackage{geometry} 

\newcommand{\HH}{{\mathbb H}}

\newcommand{\pslc}{{\mathrm{PSL}_2 (\mathbb{C})}}
\newcommand{\pslnc}{{\mathrm{PSL}_n (\mathbb{C})}}

\newcommand{\pslr}{{\mathrm{PSL}_2 (\mathbb{R})}}

\newcommand{\cp}{\mathbb{C}\mathrm{P}^1}
\newcommand{\sm}{(\mathbb{S}, \mathbb{M})}

\theoremstyle{plain}                    
\newtheorem{thm}{Theorem}[section]

\newtheorem{lem}[thm]{Lemma}

\newtheorem{defn}[thm]{Definition}

\newtheorem{prop}[thm]{Proposition}
\newtheorem{cor}[thm]{Corollary}

\theoremstyle{definition}

\newtheorem*{conj2}{Conjecture}

\theoremstyle{remark}

\title{Harmonic maps and framed $\pslc$-representations}

\author{Subhojoy Gupta}
\address{Department of Mathematics, Indian Institute of Science, Bangalore, India}
\email{subhojoy@iisc.ac.in}

\author{Gobinda Sau} 
\address{Stat-Math Unit, Indian Statistical Institute, Bangalore, India}
\email{gobindasau@isibang.ac.in}

\begin{document}

\begin{abstract}
We show that given an element $X$ of the enhanced Teichm\"{u}ller space $\mathcal{T}^\pm\sm$ and a type-preserving framed $\pslc$-representation $\hat{\rho} = (\rho,\beta)$, there is a $\rho$-equivariant harmonic map $f:\mathbb{H}^2 \to \mathbb{H}^3$ that is asymptotic to the framing $\beta$. Here, the domain is the universal cover of the punctured Riemann surface obtained from a conformal completion of $X$.  Moreover, such a harmonic map is unique if one prescribes, in addition, the principal part of the Hopf differential at each puncture. The proof uses the harmonic map heat flow.
 \end{abstract}

\maketitle

\section{Introduction}

The existence of equivariant harmonic maps from the universal cover of a Riemann surface to a symmetric space  forms is a crucial step in the proof of the celebrated non-abelian Hodge correspondence (see, for example, the discussion in \cite[\S 1]{Li-Survey}). In particular, for surface-group representations into $\pslc$, Donaldson proved the following result concerning equivariant harmonic maps into the corresponding symmetric space, namely hyperbolic $3$-space $\mathbb{H}^3$.

\begin{thm}[Donaldson, \cite{Donaldson87}] Let $X$ be a closed Riemann surface of genus $g\geq 2$ and let $\rho:\pi_1(X) \to \pslc$ be an irreducible representation. Then, there exists a $\rho$-equivariant harmonic map $f:\mathbb{H}^2 \to \mathbb{H}^3$ where the domain is identified as the universal cover of $X$. 
\end{thm}

This was generalized by Corlette in \cite{Cor92} to the case of  surface-group representations into more general semi-simple Lie groups. 
In this article, we generalize this in a different direction, namely to the context of \textit{marked and bordered surfaces} and \textit{framed} $\pslc$-representations, as introduced by Fock-Goncharov in their seminal work \cite{FG06}.

\vspace{.05in} 

Throughout this article, $\sm$ shall denote a marked and bordered surface, namely $\mathbb{S}$ shall be an oriented surface of genus $g$ and $k$ boundary components and $\mathbb{M}$ a non-empty set of marked points, such that each boundary component has at least one point. The rest of the marked points in the interior of $\mathbb{S}$  form the set $\mathbb{P} \subset \mathbb{M}$. Throughout, we shall assume that the Euler characteristic  $\chi(\mathbb{S} \setminus \mathbb{P}) <0$. 

\vspace{.05in}

To state our result, we briefly introduce a couple of relevant spaces and notions, which shall be discussed in more detail in \S2.  First, the deformation space of hyperbolic structures on $\sm$ define the \textit{enhanced Teichm\"{u}ller space} $\mathcal{T}^\pm\sm$.  An element $X\in $  $\mathcal{T}^\pm\sm$ is a (marked) hyperbolic surface possibly with cusps, or closed geodesic boundary components or crown ends, and is thus possibly incomplete; it admits a \textit{conformal completion} $\hat{X}$ which is a punctured Riemann surface homeomorphic to $\mathbb{S} \setminus \mathbb{P}$ obtained by attaching half-cylinders and half-planes along the geodesic boundary components of $X$ (see Definition \ref{ccomp}). Second, a  \textit{framed} $\pslc$-representation is a pair of a representation $\rho: \pi_1(\mathbb{S} \setminus \mathbb{P}) \to \pslc$ together with a $\rho$-equivariant map (the \textit{framing}) $\beta:F_\infty \to \cp$ that is defined on the \textit{Farey set}, the lift of $\mathbb{M}$ to the ideal boundary of the universal cover $\widetilde{X}$. Such a representation is \textit{type-preserving} if peripheral elements that are parabolic (respectively, hyperbolic) are mapping to elements in $\pslc$ that are parabolic (respectively, hyperbolic).  Finally, a $\rho$-equivariant map $f:\mathbb{H}^2 \to \HH^3$ is said to be \textit{asymptotic} to the framing $\beta$ if for any sequence  $x_n \to p\in F_\infty$, we have $f(x_n) \to \beta(p)$. Here, the domain of the map $f$ is the universal cover of the conformal completion $\hat{X}$, that naturally contains $\widetilde{X}$ (see \S2.1). 

\vspace{.1in}

\noindent We shall prove:

\begin{thm}\label{thm:main} Let $X \in \mathcal{T}^\pm\sm$ and $\hat{X}$ its conformal completion, and let $\hat{\rho} = (\rho, \beta) :\pi_1(\hat{X}) \to \pslc$ be a non-degenerate type-preserving framed representation. Then, there is a $\rho$-equivariant harmonic map $u:\mathbb{H}^2 \to \mathbb{H}^3$ that is asymptotic to the framing $\beta$.
 Moreover, $u$ is unique if we prescribe compatible principal parts of the Hopf differential  at the punctures. 
\end{thm}

For the second statement, recall that the \textit{Hopf differential} of a harmonic map $h$  is a holomorphic quadratic differential  $\text{Hopf}(h)$ defined  by the $(2,0)$-part of the pullback metric tensor $\text{Hopf}(h) = (h^\ast g)^{(2,0)}$ where $g$ is the metric in the target space. It is known from work of Wolf, Han-Tam-Treibergs-Wan ( \cite{HTTW95}, see also \cite{HanRemarks}) and Minsky (\cite{Min92}) that this carries information on the asymptotic behaviour of the harmonic map (see \S2.3). In our context, from the equivariance of the map, the Hopf differential descends to a meromorphic quadratic differential $q$ on $\hat{X}$ with poles of order at most two at the punctures corresponding to cusps and geodesic boundary components, and higher-order poles at the punctures corresponding to crowns. The principal part at a pole comprises the terms of negative degree in an expansion of $\sqrt q$ with respect to a choice of a coordinate chart (see Definition \ref{defn:ppart}).

\vspace{.05in}

In a companion article \cite{GS1} (see also the erratum \cite{GS1E}), we prove a similar existence result for harmonic maps from the complex plane $\mathbb{C}$ to $\mathbb{H}^3$; this can be thought of as a case when the domain is a disk with marked points on the boundary. 
As in \cite{GS1}, the strategy is to define a suitable $\rho$-equivariant initial map $u_0:\mathbb{H}^2 \to \mathbb{H}^3$ and use the harmonic map heat flow to converge to the desired harmonic map. One of the features of the harmonic map is that it has \textit{infinite energy} whenever the surface $\mathbb{S}$ has non-empty boundary.  We continue to rely on results of Eells-Sampson, Li-Tam, and Jiaping Wang concerning such a flow, particularly in this non-compact setting.  The construction of the initial map $u_0$, specifically in the lifts of the completion of crown ends, involves extending the arguments of \cite{GS1} to this equivariant case. The proof of the convergence of the flow, however, is different from that in \cite{GS1}, and relies on the fact that the framed representation is \textit{non-degenerate}, together with an \textit{exponential decay} of the tension field for the initial map we construct. The construction of the initial map $u_0$ uses  a recent result of Allegretti (\cite[Theorem 1.2]{allegretti2021stability}) that is essentially Theorem \ref{thm:main} in the special case when $\rho$ is a Fuchsian representation into $\pslr$, and combines prior work of the first author in \cite{GupWild} and Sagman in \cite{Sagman}. We provide a brief summary of these preliminary results in \S2, before proceeding to the proof of the main result in \S3. 

\vspace{.1in}

We expect that there is an analogue of Theorem \ref{thm:main} for framed surface-group representations into $\pslnc$ for $n>2$, where we replace $\mathbb{H}^3$ by the symmetric space $X_n =\pslnc/\text{PSU}(n)$ and where the framing takes values in the space $\mathcal{F}(\mathbb{C}^n)$  of complete flags in $\mathbb{C}^n$, which is the Furstenberg boundary of $X_n$:

\begin{conj2}
\textit{Let $X \in \mathcal{T}^{\pm}(\mathbb{S}, \mathbb{M})$ and $\hat{X}$ be as before. Given a non-degenerate type-preserving framed representation $\hat{\rho} = (\rho,\beta): (\pi_1(\hat{X}), F_\infty) \to (\pslnc, \mathcal{F}(\mathbb{C}^n))$
there exists a $\rho$-equivariant harmonic map $h:\mathbb{H}^2 \to \mathbb{X}_n$
that is asymptotic to $\beta$. Moreover, the map is unique if the principal parts of the Hopf differential at the punctures are prescribed.}
\end{conj2}

We note that the analytic component of this paper, particularly the existence of the harmonic map heat flow, would carry through to the case when the target is replaced by the non-positively curved space $\mathbb{X}_n$.  However, our construction of the initial map in this paper crucially uses the geometry of $\mathbb{X}_2 = \mathbb{H}^3$; in particular, the notion of a ``pleated plane" implicitly plays a role in our argument (see \S3.1), and it has  no good analogue in higher-rank symmetric spaces. 

\vspace{.05in}

Although equivariant harmonic maps to $\mathbb{X}_n$ with the Hopf differential having poles of higher order do appear in the context of the ``wild" non-abelian Hodge correspondence (see, for example \cite{BiqBoalch} and \cite{Mochi}), the conjecture above (a special case of which is our main theorem) provides a new perspective: namely, for fixed data of the principal parts at the punctures, there is a bijective correspondence between these equivariant harmonic maps on one hand, and a suitable subspace (as in, compatible with the principal parts) of the  moduli space of framed representations, on the other.

\vspace{.1in}

\subsection*{Acknowledgements.} Several parts of the work in this paper and its prequel \cite{GS1} are contained in the PhD thesis of GS \cite{thesisGobinda}, written under SG's supervision. The authors thank Qiongling Li for her help and advice, and to Nathaniel Sagman and Mike Wolf for their comments and interest in this project. This work was supported by the Department of Science and Technology, Govt.of India grant no. CRG/2022/001822, and by the DST FIST program - 2021 [TPN - 700661].

\vspace{.1in}

\section{Preliminaries}

\subsection{Enhanced Teichm\"{u}ller space}
In what follows $\sm$ shall denote a marked and bordered surface, with $\mathbb{S}$ a compact surface of genus $g\geq 0$ with $k\geq 0$ boundary components,  and $\mathbb{M}$ is a non-empty set of marked points with  at least one marked point on each boundary component, together with a (possibly empty) subset $\mathbb{P}$ of interior marked points. We shall assume throughout that $\chi(\mathbb{S} \setminus \mathbb{P}) < 0$.  The companion paper \cite{GS1} shows that the result is also true in the non-equivariant setting, when $\mathbb{S}$ is the disk (\textit{i.e.} has genus $g=0$ and $k=1$), $\mathbb{P} = \emptyset$,  and $|\mathbb{M}| \geq 3$.

\vspace{.05in}

The \textit{enhanced Teichm\"{u}ller space} $\mathcal{T}^\pm\sm$  is the deformation space of marked hyperbolic structures on $\sm$, where the points in $\mathbb{P}$ are cusps or boundary components, and the $k$ boundary components form crown ends (where each marked point on a boundary component becomes a \textit{boundary cusp}, and each boundary arc between such boundary marked points is a bi-infinite geodesic line). In addition, there is an additional signing ($\pm$) associated to each geodesic boundary component, which can be thought of as a choice of orientation. 

We refer to Allegretti \cite[\S3]{allegretti2021stability} for a more precise definition, and \cite[Proposition 3.5]{allegretti2021stability}  for the parametrization $\mathcal{T}^\pm\sm \cong \mathbb{R}^N$ where $N=6g-6+\sum_{i=1}^k(n_i + 3)$ where $n_i$= $\#\{$marked points on the $i$-th boundary component$\}$.

\vspace{.05in} 

\noindent We introduce the following notion:

\begin{defn}[Conformal completion]\label{ccomp} Given a hyperbolic surface $X \in \mathcal{T}^\pm\sm$, its \textit{conformal completion} $\hat{X}$ is the punctured Riemann surface obtained by performing an \textit{infinite grafting} along each boundary component of $X$, that amounts to:
\begin{itemize}
    \item[(i)] attaching a half-infinite Euclidean cylinder to any boundary component that is a closed circle, and 
    \item[(ii)] attaching a Euclidean half-plane to any boundary component that is a bi-infinite geodesic line,
\end{itemize}
where the attaching is via an isometry on the circle in case (i) and the line in case (ii).   (See Figure \ref{fig1}.)

A puncture corresponding to (i) will be referred to as a ``cylinder-end", while a ``crown-end" arises from a crown of $X$ by applying (ii) to each geodesic line of the crown boundary. 
\end{defn}

\noindent \textit{Remarks.} (i) For the notion of ``infinite grafting", see \cite[\S3]{GupWild} and \cite[\S4.2-3]{SGMahangrafting}. In particular, see \cite[Lemma 4.3]{SGMahangrafting} for the fact that the conformal completion of a crown end is conformally a punctured disk.

\begin{figure}
\begin{center} 
\includegraphics[scale=.4]{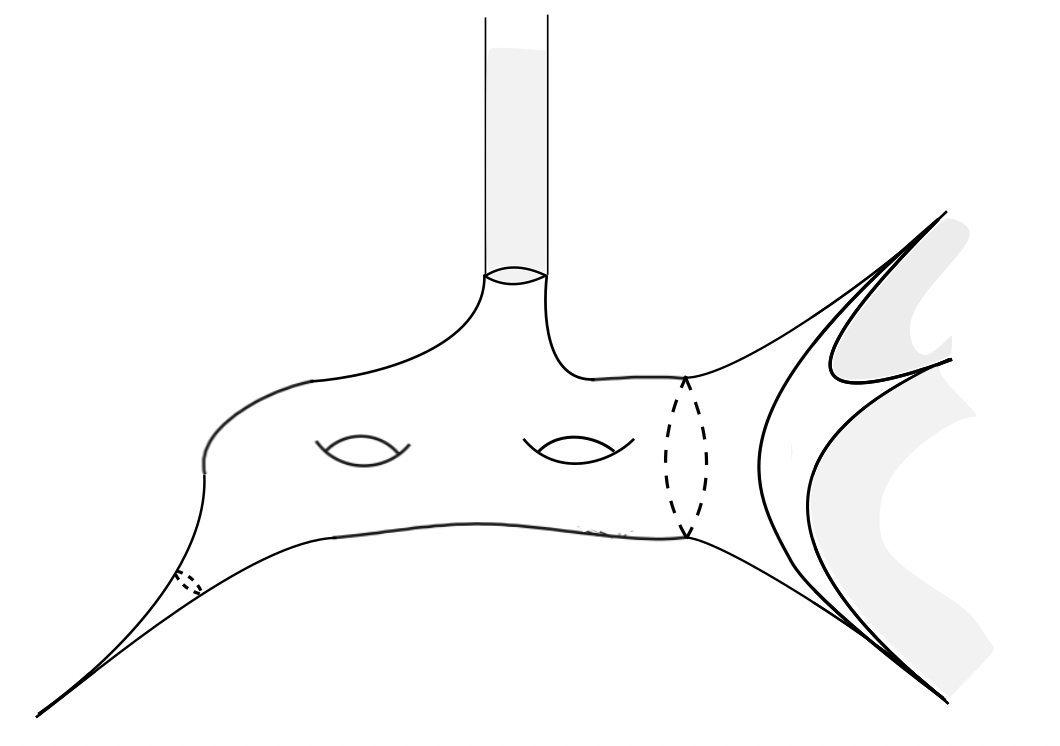}
\caption{The conformal completion of a hyperbolic surface $X$ with  one cusp, one closed geodesic boundary and one crown end involves adding Euclidean half-planes and half-cylinders (shown shaded). }
\label{fig1}
\end{center}
\end{figure} 

\vspace{.05in}

\noindent (ii) The universal cover of a conformal completion $\hat{X}$ is conformally the upper half-plane $\HH^2$, however it shall be useful to consider it as the surface $\mathcal{H}$ obtained by attaching Euclidean half-planes to each geodesic boundary component of the universal cover $\widetilde{X}$ of the hyperbolic surface $X$. In particular, $\widetilde{X} \subset \mathcal{H}$, and one can make sense of a sequence of points in $\mathcal{H}$ converging to a point in the Farey set $F_\infty$.

\vspace{.05in}

Note that the punctures of the conformal completion $\hat{X}$ are of two kinds -- one that arises from the interior marked points of $\sm$, which are either cusps or geodesic boundary components of $X$, and the other that arises from each crown end of $X$ after attaching half-planes as in (ii) of the definition above.This dichotomy is useful to keep in mind, as these two types of punctures will need to be handled differently.

\subsection{Framed representations}

An \textit{ideal triangulation} of $\sm$ is a triangulation of the punctured surface $\mathbb{S} \setminus \mathbb{P}$ with vertices at the points of $\mathbb{M}$.  If one equips $\sm$ with a hyperbolic structure, i.e.\ considers an element of $\mathcal{T}^\pm\sm$, then the lift of such a triangulation is realized by an ideal triangulation on the universal cover $\mathcal{H} \cong \mathbb{H}^2$ (\textit{c.f.} Remark (ii) after Definition \ref{ccomp}), with ideal vertices in the \textit{Farey set}, which is the set of lifts of the points of $\mathbb{M}$ to the ideal boundary of $\widetilde{X} \subset \mathcal{H}$.

\vspace{.05in}

We have already defined a {framed $\pslc$-representation} in the Introduction; we note, however, the following:

(i)  It is immediate  from the $\rho$-equivariance of the framing that the framing assigns to a lift of $p \in \mathbb{P}$ the fixed point of the corresponding conjugate of $\rho(\gamma)$, where $\gamma$ is the peripheral curve around $p$. 

(ii) A marked hyperbolic surface $X$ that is an element of the enhanced Teichm\"{u}ller space determines (an equivalence class of) a framed representation $(\rho, \beta)$ that is \textit{Fuchsian}: in that case the image of $\rho$ is a discrete group $\Gamma$ of $\pslr =\text{Isom}^+(\HH^2) $, and there is an invariant convex subset $\mathcal{C} \subset \mathbb{H}^2$ such that $\mathcal{C}/\Gamma = X$. Moreover, the lifts of $\mathbb{M}$ can then be identified with points in the ideal boundary of $\mathcal{C}$; this determines the framing $\beta$.

\vspace{.05in}

If $X\sm$ is the space of framed $\pslc$-representations, then the \textit{moduli space of framed $\pslc$ representations} is \[\widehat{\chi}\sm=X\sm/\text{PSL}_2(\mathbb{C}),\] where the quotient is by the action $A.(\rho,\beta_{\rho})= (A\rho A^{-1},A.\beta_{\rho})$ for $A\in \text{PSL}_2(\mathbb{C})$. 

\vspace{.1in} 
In \cite{FG06}, Fock-Goncharov  introduced moduli spaces of framed $\pslnc$-representations (for any $n\geq 2$), and provided a birational parametrization of these spaces (see Theorem 1 of \cite{FG06}). In the case of $\pslc$, their theorem provides the birational map $\widehat{\chi}\sm \cong (\mathbb{C}^\ast)^E$ where $E$ is the number of edges of an ideal triangulation $T$ of $\sm$. Briefly, the Fock-Goncharov coordinate of an edge is the complex cross-ratio of the four points in $\cp$ determined by the image under the framing of the vertices of the two adjacent ideal triangles. Representations with real and positive Fock-Goncharov parameters are precisely Fuchsian representations, corresponding to $\mathcal{T}^\pm \sm$. 

\vspace{.05in} 
 This parametrizes an open dense set of $\widehat{\chi}\sm$, namely, the subset where each such cross-ratio is well-defined.  By \cite[Theorem 9.1]{All-Bri}, this subset is precisely the space of non-degenerate framed representations, as defined below (see also \cite[Definition 4.3]{All-Bri}):

\begin{defn}[Non-degenerate framed representation] \label{def:deg}
    A framed representation $\hat{\rho} = (\rho, \beta)$ is \textit{degenerate} if one of the following holds:
\begin{enumerate}
    \item The image of $\beta$ is a single point $\{p\}$ and the monodromy around each puncture is parabolic with fixed point $p$ or the identity.

    \item The image of $\beta$ has two points $\{p,q\}$ and the monodromy around each puncture  fixes both $p$ and $q$.

\item  There is a boundary segment $I$ of $\mathbb{S}$ joining two consecutive marked points $p_1,p_2$, and a lift $\tilde{I}$, such that its endpoints are given by lifts $\tilde{p_1},\tilde{p_2} \in F_\infty$ satisfying $\beta(\tilde{p_1}) = \beta(\tilde{p_2})$. 

\end{enumerate}

    The framed representation is said to be non-degenerate if it is not degenerate.  

\end{defn}

In what follows we denote $\Pi := \pi_1(\mathbb{S} \setminus \mathbb{P})$. Recall that an element $A\in \text{PSL}_2(\mathbb{C})$ is called \textit{semi-simple} if it is elliptic or hyperbolic. We note the following: 

\begin{lem}\label{eqivuse}
 Let $\rho:\Pi\rightarrow \text{PSL}_2(\mathbb{C})$ be a non-degenerate type-preserving framed representation. Then there exists $\gamma,\gamma'$ in $\Pi$ such that $\rho(\gamma),\rho(\gamma')$ are two different semi-simple elements with different axes.
\end{lem} 
\begin{proof}
It follows from the definition of type-preserving in the Introduction, that $\rho(\gamma) \neq I$ for any peripheral element $\gamma \in \Pi$; in fact $\rho(\gamma)$ is a parabolic (resp. loxodromic) if $\gamma$ is parabolic (resp. hyperbolic). In particular, the non-degenerate representation $\rho:\Pi\rightarrow \text{PSL}_2(\mathbb{C})$ has \textit{no apparent singularity}. 
Following the proof of Proposition $3.1$ in \cite{MR4286048}, we conclude that a non-degenerate framed representation without apparent singularities is indeed a non-degenerate representation in the sense of Definition $2.4$ of \cite{MR4286048}. If $\rho(\Pi)$ contains only parabolic elements with a global fixed point in $\mathbb{CP}^1$, then by definition, $\rho$ is degenerate in the sense of Definition 2.4 of \cite{MR4286048}, a contradiction. Let $\alpha$ and $\delta$ be two different parabolic elements in $\rho(\Pi)$ with different fixed points. Without loss of generality, assume $\alpha$ fixes $\infty$,  with $\alpha=\begin{bmatrix}
1&1\\
0&1	
\end{bmatrix}$ and $\delta=\begin{bmatrix}
a&b\\
c&d
\end{bmatrix}$.
Then $\text{tr}^2(\alpha^n\delta)=(a+d+nc)^2$ for all $n\in\mathbb{N}$. For any parabolic element $\gamma\in \text{PSL}_2(\mathbb{C})$, we have $ \text{tr}^2(\gamma)=4$. Therefore $\alpha^n\delta$ is parabolic if and only if $c=0$ implying $\alpha$ and $\delta$ have a unique common fixed point, a contradiction. Therefore there exists a semi-simple element $T$ in $\rho(\Pi)$ with an axis $L$. If all elements of $\rho(\Pi)$ fixes the ideal points $p_{+}$ and $p_{-}$ point-wise, then $\rho$ is degenerate. Therefore, there is an element $g$ that moves the ideal points $p_{+}$ and $p_{-}$, then $gTg^{-1}$ is a semi-simple element of the same type as $T$ with different axes. This completes the proof.
\end{proof}

\subsection{Harmonic maps and their asymptotic behaviour}

We provide definitions in a more general context: Given two Riemannian manifolds $(M,g)$ and $(N,h)$, a ($C^2$-smooth) \textit{harmonic map} $u:M\to N$ is a critical point of the energy functional
\begin{equation}\label{eq:efunc}
E^U(u)=\frac{1}{2}\int_{U}||du||^2 dvol_g
\end{equation} 
on every relatively compact open subset $U$ of the domain. Here, the function in the integrand is the called the {energy density} of $u$, often denoted by $e(u)$. Alternatively, a harmonic map is characterized by the fact that its \textit{tension field} \begin{equation}\label{eq:tf}
    \tau(u)=\text{Tr}_g(\nabla du)
\end{equation} vanishes (see Lemma 2.3 of \cite{GS1}). 

\vspace{.05in} 

In the case the domain is a surface $\Sigma$, as in this paper, the harmonicity of the map $u$ depends only on the conformal class of the domain metric $g$.  As mentioned in \S1, the \textit{Hopf differential} of such a harmonic map is the holomorphic quadratic differential $\text{Hopf}(u) = (u^\ast h)^{(2,0)}$, where $h$ is the metric in the target manifold. This plays an important role in determining the asymptotic behaviour of the harmonic map, as we shall now explain.  

\vspace{.05in} 

Recall that a holomorphic quadratic differential  $q$ induces \textit{horizontal} and \textit{vertical foliations} on the domain surface $\Sigma$, smooth away from the zeros of $q$, and induces the $q$-metric on $\Sigma$, which is a  flat metric  having singularities at the zeros. (See, for example, \cite[Definitions 2.7 and 2.8]{GS1}.)

The following result summarizes the relation of the singular-flat geometry of the Hopf differential to the behaviour of a harmonic map  far from the zeros of the differential -- see \cite[Theorem 4.2]{Min92}, and in our companion paper \cite{GS1}, see Proposition 2.10, and the remark following Proposition 2.13. (Note that the factor of $2$ in Minsky's paper is avoided in the latter paper by considering a scaling of the $q$-metric by a factor $4$.)

\begin{prop}\label{prop:est} 
    Let $u:\Sigma \to \mathbb{H}^3$ be a harmonic map with Hopf differential $q$. Let $\Sigma$ be equipped with the $4q$-metric. Then for $R\gg 0$, we have the estimate $$ \lVert u (x) - \Pi(x)\rVert_{C^1} = O(e^{-\alpha R})$$
    where $\Pi$, defined in a neighborhood of $x$, is a  map to a geodesic segment in $\mathbb{H}^3$ that collapses the vertical foliation
    and is an isometry along the leaves of the horizontal foliation. 
    Here $\alpha>0$ is a universal constant, and $R$ is the distance of $x$ from the zero set of $q$. 
\end{prop}

\vspace{.05in} 

In this paper, the harmonic maps $u:\mathbb{H}^2 \to \mathbb{H}^3$  we shall consider would moreover be \textit{equivariant} with respect to a representation $\rho:\Gamma \to \pslc$. Here,  the domain $\HH^2$ is the universal cover of a punctured Riemann surface $\hat{X}$ equipped with the action of the deck-group $\Gamma$, and  $\rho(\Gamma)$ acts by isometries on the target $\HH^3$. Namely, the map would satisfy $u(\gamma \cdot x) = \rho(\gamma)\cdot u(x)$
for all $x\in \HH^2$ and $\gamma \in \Gamma$. The Hopf differential $\text{Hopf}(u)$ would then be $\Gamma$-invariant, and descend to a holomorphic quadratic differential $q$ on $\hat{X}$. We shall moreover require that  $q$ extends to \textit{finite order} poles at the punctures of $\hat{X}$. We shall now describe some further data associated with these poles, namely their \textit{principal parts}, that are defined with respect to a choice of a coordinate neighborhood around each puncture. See also \cite[\S2.5]{allegretti2021stability} for a discussion. 

\vspace{.05in} 

\begin{defn}[Principal differential]
A principal differential on $\mathbb{D}^*$
of order $n\geq 3$ is defined to be a meromorphic differential of the form
\begin{equation}\label{eq:ppart}
z^{-\epsilon}\left(\frac{{\alpha}_r}{z^r}+ \frac{{\alpha}_{r-1}}{z^{r-1}}+\cdot+\frac{{\alpha}_1}{z}\right) dz
\end{equation} 
on $\mathbb{D}^*$ where $r=\lfloor \frac{n}{2}\rfloor$ and $\epsilon =0$ if $n$ is even and $1/2$ if $n$ is odd.  The space of such principal differentials will be denoted by $Princ(n)$.
\end{defn}

\begin{defn}[Principal part, Residue]\label{defn:ppart}
Given a meromorphic quadratic differential q
on $X$ with a pole of order $n\geq 3$ at $p$, and a choice of coordinate disk $U\cong \mathbb{D}$
around the point $p$, we define its principal part $Pr(q)$ to be the principal differential (up to a choice of sign) such that $\sqrt{q|_U}-Pr(q)$ is  integrable on $U$. If  the order $n$ is even, the (analytic) residue of the principal part is defined to be residue of $Pr(q)$, namely, the coefficient $\alpha_1$ in the expression \eqref{eq:ppart} (again, defined up to a choice of sign).

In the case that $q$ has a pole of order $n=2$ at $p$, with a leading term $\frac{a}{z^2} dz^2$, then the principal part  is defined to be $\pm \frac{\sqrt a}{z} dz$  and its residue is defined to be $\pm 4\pi i \sqrt{a}$. (See \cite[\S2.5]{allegretti2021stability}.) 

\end{defn}
\vspace{.05in}

\noindent We shall also need the following notion (see \cite[\S4.3]{allegretti2021stability}):

\begin{defn}[Compatible principal part - I]\label{defn:comp} A principal part $\mathbb{P}$ for a pole of order $m>2$ that is even, is said to be \textit{compatible} with a hyperbolic crown end $C$ if the crown end has $m-2$ boundary cusps, and its metric residue, as defined in \cite[Definitions 2.8,2.9]{GupWild}, equals the real part of the residue of $\mathbb{P}$.  If the order $m>2$ is odd, then the principal part is  compatible with $C$ if the crown end has $m-2$ boundary cusps.

If $\mathbb{P}$ is for a pole of order $m\leq 2$, then it is said to be compatible with a geodesic boundary of length $L$ or a cusp (when $L=0$) if $L^2 = 16\pi^2 \lvert a\rvert \sin^2(\theta/2)$ where the principal part (or leading coefficient) of $\mathbb{P}$ is $ae^{i\theta}$.
\end{defn}

\vspace{.05in}

We can now state the following existence and uniqueness result from \cite[Theorem 1.2]{allegretti2021stability}  for framed $\pslr$-representations which are Fuchsian, that extends previous work of \cite{GupWild} and \cite{Sagman}:

\begin{thm}[Allegretti]\label{thm:all} Given an element  $Y$ of $\mathcal{T}^\pm\sm$, and a punctured Riemann surface $\hat{X}$, there is a harmonic diffeomorphism $F : \hat{X} \rightarrow Y \setminus \partial Y $ that is homotopic to the marking, and whose Hopf differential is a  meromorphic quadratic differential with finite order poles at the punctures. Moreover, such a map is unique if one prescribes the principal part at each pole, that is compatible with the corresponding end in $Y$, in the sense defined above.  
\end{thm}

\textit{Remark.} The fact that the compatibility condition is necessary, is a consequence of the asymptotic behaviour of the harmonic map (Proposition \ref{prop:est}) -- in \cite{GupWild} see Proposition 2.29 and the remark following it.

\vspace{.05in}

We shall also record some results on the behaviour of $\mathbb{Z}$-equivariant harmonic maps from the universal cover $\HH^2$ of a punctured disk $\mathbb{D}^\ast$ to $\HH^3$ such that the  Hopf differential of such a map  descends to a higher-order pole at the puncture of $\mathbb{D}^\ast$. 
Here, the infinite cyclic group $\mathbb{Z}$ can be thought of as acting by deck-translations on the domain, and by powers of some element $A \in \pslc$ in the target. This can be thought of as a $\mathbb{Z}$-equivariant version of the asymptotic behaviour of harmonic maps from $\mathbb{C}$ to $\HH^3$ with polynomial Hopf differential, studied in \cite{GS1}.

\vspace{.05in}

For this, we introduce the notion of ``chains of geodesics", that can be thought of as $\mathbb{Z}$-invariant versions of twisted ideal polygons (considered in \cite{GS1}):

\begin{defn}[Chains]\label{defn:chain} Given an $m\geq 1$ and a loxodromic element $A \in \pslc$, a \textit{chain} of geodesics in $\HH^3$ is defined by an ordered set $M_\infty$ of ideal points on $\cp$ that is invariant under the infinite cyclic group $\mathbb{Z} \cong \langle A\rangle$. Each successive pair of ideal points in $M_\infty$ is connected by a bi-infinite geodesic, and the union of these forms the chain. Here, $M_\infty$ is defined by taking the ordered set of $m$-points $M_0= \{p_1, p_2,\ldots p_m\}$ and considering its $\mathbb{Z}$-orbit; namely, $M_\infty = \bigcup_{n \in \mathbb{Z}} \{ A^n \cdot p_1, A^n \cdot p_2, \cdots, A^n \cdot p_m\}$. We shall refer to a chain having $m$ points in $M_0$ (the fundamental domain of the $\mathbb{Z}$-action) as an $m$-chain. 
\end{defn} 

\textit{Remark.} In the case that $A \in \pslr$ is a hyperbolic element, and the chain $\mathcal{C}$ lies on a totally geodesic copy of $\HH^2$, say $H \subset \mathbb{H}^3$, then the quotient $H/\langle A\rangle$ has a crown end with $m$ geodesics obtained from $\mathcal{C}$.

\vspace{.05in}

\noindent We also introduce the following terms associated with chains that will be used subsequently:


\begin{defn}[Straightening]\label{defn:stn}
Let $\mathcal{C}$ be a chain in $\HH^3$  as above, with a corresponding ordered set of ideal points $M_\infty$. Any pair of geodesics  $g_1, g_2$ incident at $p \in M_\infty $ lies on pair of totally-geodesic copies of the hyperbolic plane, that intersect along a line $\ell$; an elliptic rotation with axis $\ell$ can be applied to points of $M_\infty$ with index less than $p$ such that in the resulting chain the geodesics $g_1, g_2^\prime$ are co-planar and lie on either side of $\ell$. Doing this equivariantly, i.e., applying the same elliptic element to each element of the $\mathbb{Z}$-orbit of $p$, and repeating this for each cusp in $M_0$, we obtain a ``straightened" chain $\mathcal{C}^\prime$ that lies on a totally geodesic copy of $\HH^2$ and is a lift of a hyperbolic crown (\textit{c.f.} the previous remark).  
\end{defn} 

\begin{defn}[Metric residue]\label{defn:resch} For an $m$-chain $\mathcal{C}$ where $m$ is even, we define the ``metric residue" of the chain to be the real number $\alpha$ (defined up to sign) obtained as follows: truncate the geodesics in $\mathcal{C}$ by choosing a $\mathbb{Z}$-invariant collection of horoballs at each point in $M_\infty$, and let $l_1,l_2,\ldots l_m$ be the $m$ lengths of these truncated geodesics (up to the $\mathbb{Z}$-action). Then $\alpha :=  \pm \sum\limits_{i=1}^m (-1)^i l_i$. 
Alternatively, we can define the metric residue to be  the metric residue of the hyperbolic crown associated to the straightening of $\mathcal{C}$ as defined above. (The metric residue of a crown was defined in  \cite[Definitions 2.8,2.9]{GupWild}.)
\end{defn}

\begin{defn}[Compatible principal part --II]\label{defn:comp2}
In continuation of Definition \ref{defn:comp}, we will say that a principal part $\mathbb{P}$ for a pole of order $n>2$ that is even, is said to be compatible with an $m$-chain $\mathcal{C}$ if $m=n-2$ and the metric residue of $\mathcal{C}$ (as defined above) equals the real part of the residue of $\mathbb{P}$.  In the case that $n$ is odd, we drop the latter requirement.
\end{defn} 

\vspace{.05in}

\noindent The main observation is that as a consequence of the estimates in Proposition \ref{prop:est}, we can determine the asymptotic behaviour of a harmonic map near a pole of the Hopf differential). In the statements below, $U \cong \mathbb{D}^\ast$ is a neighborhood of the pole, and $\widetilde{U} = \{ z \in \HH^2\ \vert\ \text{Im}(z)>1 \}$ its universal cover, such that the covering map $p: \widetilde{U} \to U$ is given by $p(z) = e^{2\pi i z}$. In the following two lemmas, we shall also assume that the domain is equipped with the Euclidean metric. 

\begin{lem}[Asymptotic behaviour: higher-order pole]\label{map2chain} 
Let $h:\widetilde{U} \to \HH^3$ be a $\mathbb{Z}$-equivariant harmonic map such that $\text{Hopf}(h)$ descends to a holomorphic quadratic differential $q$ on $U \cong \mathbb{D}^\ast$ having a pole of order $n\geq 3$ at the puncture. Then, the map $h$ is asymptotic to an  $(n-2)$-chain of geodesics $\mathcal{C}$. Moreover, if $n$ is even, and $\mathbb{P}$ is the principal part of the pole, then the real part of the residue of $\mathbb{P}$ equals the metric residue of $\mathcal{C}$.
\end{lem}
\begin{proof}
On the neighborhood $U$ of  the pole of order $n\geq 3$, the  meromorphic Hopf differential $q$ induces a flat metric that can be decomposed into a cyclic collection of $(n-2)$ horizontal and  $(n-2)$ vertical half-planes around the puncture, alternating and overlapping along quarter-planes. Here, a horizontal (resp. vertical) half-plane acquires the Euclidean metric induced by $q$ and is foliated by bi-infinite leaves of the horizontal (resp. vertical) foliation. See, for example, \cite[pg. 79]{GupWolf}.  Passing to the universal cover $\widetilde{U}$ , this induces a chain of overlapping horizontal and vertical half-planes invariant under the deck-group $\mathbb{Z}$, such that the fundamental domain has exactly $(n-2)$ horizontal and $(n-2)$ vertical half-planes. 

In each horizontal half-plane $H$, the estimate in Proposition \ref{prop:est} together with the Canoeing Lemma (see \cite[Lemma 2.12]{GS1}) implies that a bi-infinite horizontal leaf far from the boundary is mapped by $h$  almost-isometrically to a parametrized geodesic line $\ell$ in $\HH^3$; that is, the restriction of $h$ to $H$ is asymptotic to $\ell$.   A pair of successive horizontal half-planes $H_1, H_2$  have an overlapping vertical half-plane $V$; by Proposition \ref{prop:est} the $h$-images of leaves on $V$ that are far from the boundary collapse as the distance from the boundary tends to infinity; this implies that the geodesic lines $\ell_1$ and $\ell_2$ corresponding to $H_1$ and $H_2$ share an ideal endpoint (\textit{c.f.} \cite[Lemma 3.3]{HTTW95}).  Thus, the chain of half-planes maps to a chain of geodesics in the sense of Definition \ref{defn:chain}; since there are $(n-2)$ such planes in a fundamental domain, it is an $(n-2)$-chain. 

The statement about the principal part follows from the argument in \cite[Proposition 2.29]{GupWild}, and we refer the reader to that for details: first, we can exhaust the punctured-disk $U$ by a sequence of regions $P_i$  bounded by polygonal boundaries, comprising  alternating horizontal and vertical segments of lengths  $L_i \pm O(1) \to \infty$. Proposition \ref{prop:est} applies, and since the images of the vertical sides have lengths tending to zero, the alternating sum of the lengths of the images of the horizontal sides tends to the metric residue of the chain $\mathcal{C}$ (\textit{c.f.} Definition \ref{defn:resch}). However, this metric residue is the real part of the principal part by \cite[Lemma 2.28]{GupWild}, and we are done.    
\end{proof}

\noindent At  a pole of order two  of the Hopf differential, we note the following analogous result essentially from \cite{Sagman}:

\begin{lem}\label{map2}
Let $h:\widetilde{U} \to \HH^3$ be a $\mathbb{Z}$-equivariant harmonic map where the action on the target space is by iterates of a loxodromic element $A\in \pslc$. Assume that $\text{Hopf}(h)$ descends to a holomorphic quadratic differential $q$ on $U \cong \mathbb{D}^\ast$ having a pole of order $2$ at the puncture. Then, the map $h$ is asymptotic to the  geodesic line $\ell$ in $\HH^3$ that is the axis of $A$. Moreover, if $L$ is the length of $\ell/\langle A\rangle$, then $L^2 = 16\pi^2 \lvert a\rvert \sin^2(\theta/2)$ where the principal part of $\mathbb{P}$ is $ae^{i\theta}$.    
\end{lem}
\begin{proof}
    A quadratic differential on $U \cong \mathbb{D}^\ast$ with a pole of order $2$ at the puncture pulls back to a constant quadratic differential on the half-plane $\widetilde{U}$. As before, Proposition \ref{prop:est} and the Canoeing lemma then implies the harmonic map $h\vert_{\widetilde{U}}$ is asymptotic to a collapsing map to a geodesic line $\ell$.  (See also \cite[Lemma 5.4]{Sagman}.) The statement about the principal part follows from the computation in \cite[Lemma 5.3]{Sagman}.
\end{proof}

\noindent Finally, for a pole of order one of the Hopf differential, we note the following consequence of \cite[Lemma 4.8]{allegretti2021stability}. Here, the that the domain  $\widetilde{U}$ is equipped with the  restriction of the hyperbolic metric on $\HH^2$ such that $U$ is a hyperbolic cusp at the puncture. 

\begin{lem}\label{map3}
Let $h:\widetilde{U} \to \HH^3$ be a $\mathbb{Z}$-equivariant harmonic map where the action on the target space is by iterates of a parabolic element $A\in \pslc$. Assume that $\text{Hopf}(h)$ descends to a holomorphic quadratic differential $q$ on $U \cong \mathbb{D}^\ast$ having a pole of order $1$ at the puncture. Then, the map $h$ is asymptotic to the map $c$ that embeds $\widetilde{U}$ as a horodisk in a totally-geodesic plane in  $\HH^3$. 
\end{lem}
\begin{proof}
Using the upper half-plane model for both domain and range, the map $c:\widetilde{U} \to \HH^2$ is defined by $c(x,y) = (x,y)$. When the domain is equipped with the restriction of the hyperbolic metric, a straightforward computation verifies that the total energy $E(c) <\infty$ and $\text{Hopf}(h)$ descends to a quadratic differential with a pole of order one on $U \cong \mathbb{D}^\ast $. (See also \cite[Lemma 4.5]{allegretti2021stability}.) The second statement then follows from \cite[Lemma 4.8]{allegretti2021stability}. 
\end{proof}

\subsection{Harmonic map heat flow}

Continuing with the more general set-up, let $M,N$ be two complete Riemannian manifolds, and let  $u_0: M \to N$ be an initial map. 

The \textit{harmonic map heat flow} is the following nonlinear PDE: 
\begin{equation}\label{heat1}
\begin{split}
\frac{\partial u}{\partial t}&=\tau(u(x,t))\\
 u(x,0)&=u_0(x),  
\end{split}
\end{equation}
where the tension field $\tau(u(x,t))$ was defined in \eqref{eq:tf}). We shall often use the notation $u_t(x) := u(x,t)$ to think of this flow as a family of maps parametrized by time $t\in [0,\infty)$.

This can be thought of as the gradient flow for the energy functional \eqref{eq:efunc}. In the seminal work \cite{EellsSampson}, Eells-Sampson showed that in the case that the manifolds $M,N$ are compact, and the sectional curvatures of $N$ are non-positive, the flow exists for all time and converges to a harmonic map.

\vspace{.05in}

In the case when $M,N$ are non-compact, the following existence-and-uniqueness result for solutions to \eqref{heat1} was proved in \cite[Theorem 3.1]{JiapingWang} :

 \begin{thm}[Long-time solution]\label{wang}
Let $M$ and $N$ be two complete Riemannian manifolds such that the sectional curvatures $K_N\leq0$. Let $u_0:M \rightarrow N $ be a $C^2$ map. If $$b(x,t) = \left(\int_MH(x,y,t)|\tau(u_0)|^2(y)dy\right)^\frac{1}{2}$$ is finite for all $(x,t)\in M\times(0,\infty)$ where $H(x,y,t)$ is the heat kernel of $M$, then $\eqref{heat1}$ has a solution $u(x,t)$ defined for all $(x,t)\in M\times(0,\infty)$, that satisfies the tension field bound $|\tau(u)(x,t)|\leq b(x,t)$. Moreover, if $N$ is simply-connected, and for any $T>0$, the integral  $\int_{0}^{T}\int_\mathbb{C} e^{-cr^2(x)}b^2(x,t)dxdt<\infty$ for some $c>0$, the solution is unique.  (Here $r(\cdot)$ is the distance function from a basepoint.) 
\end{thm}

We shall apply this to our setting where $M = \HH^2$ and $N = \HH^3$.  Note that the above result does not guarantee that the long-time solution \textit{converges} to a harmonic map; a simple example is given as follows:  Let $u_0:\mathbb{H}^2\rightarrow\mathbb{H}^3$ be defined by $u_0(x,y)=(x,0,t_0)$ where we are using the upper half-plane model for the domain and upper half-space model for the target. Then one can check that $u_t(x,y)=(x,y,\sqrt{2y^2t+t_0^2})$ is a solution to the harmonic map heat flow \eqref{heat1}, with initial map $u_0$. A straightforward computation shows that the  tension field of $u_0$  is uniformly bounded, so the hypotheses of Theorem \ref{wang} hold; however, clearly as $t\rightarrow\infty$, the maps $u_t$ do \textit{not} converge.

\vspace{.05in}

In Donaldson's work mentioned in the Introduction, this issue of convergence was resolved using the fact that the representation being considered was irreducible.  
In our case, we first prove a \textit{one-point} convergence (see \S3.3) using the fact that our framed representation is {non-degenerate}, and then prove a global distance bound in \S3.4 by crucially using an exponential decay of the tension field of the initial map (see Lemma \ref{u0bd}).

\subsection{Some analytic preliminaries}
In this section we state a few results from the analysis of PDEs that we shall use in this paper. 

\vspace{.05in}

First, the following general result is from \cite{SchoenYau79},  see also \cite[Lemma 3.22]{GS1} for a proof. Here, recall that if $(M,g)$ is a  Riemannian manifold, then a function $u:M \to \mathbb{R}$ is said to be a \textit{subsolution of the heat equation} if \[\left(\Delta_g-\frac{\partial}{\partial t}\right)u\geq 0\] where $\Delta_g$ is the Laplacian on $(M,g)$.

\begin{lem}[\cite{SchoenYau79}]\label{sy}
Let $(M,g)$ and $(N,h)$ be two Riemannian manifolds such that $N$ is simply connected non-positively curved manifold. Assume that $\Omega \subset M$ is an open subset and $v,w:\Omega \times [0,\infty)\rightarrow N$ are two solutions of the harmonic map heat flow \eqref{heat1}. Then the distance function $d_N(v,w)$ is a subsolution of the heat equation. 
\end{lem}

\noindent \textit{Remark.} A special case of the above lemma is when $v$ or $w$ is a harmonic map on $\Omega$, independent of $t$, i.e. a \textit{stationary} solution of \eqref{heat1}.

\vspace{.05in}

\noindent Recall the following Parabolic Maximum principle for manifolds with boundary:

\begin{prop}\label{boundary}
Let $M$ be a compact manifold possibly with boundary $\partial M$, and let $u:M \times [0,\infty) \to \mathbb{R}$ be a subsolution of the heat equation on $M$.
Then for any $T>0$,  \[\underset{M\times [0, T)}{\mathrm{sup}}u=\underset{M\times \{0\}\cup\partial{M}\times[0, T)} {\mathrm{sup}}u.\]

(The domain of the supremum on the RHS is called the ``parabolic boundary".) 

\end{prop} 

\vspace{.05in}

For \textit{non-compact} manifolds, we note the following version of the parabolic maximum principle; for a proof, see \cite[Lemma 2.1]{JiapingWang} (where it is attributed to Li) or \cite[pg. 273]{MR2274812}.

\begin{prop}\label{mpr}
Let $(M,g)$ be a complete Riemannian manifold.
If $G(x,t)$ is a weak subsolution of the heat equation defined on $M\times [0,T]$ and $G(x,0)\leq 0$ for any $x\in M$, then $G(x,t)\leq 0$ for $(x,t)\in M\times [0,T]$ provided $\int_{0}^{T}\int_{M} e^{-cr^2(x)}G^2(x,t)dxdt<\infty$ for  some $c>0.$ 
\end{prop}

We have used the above statement in \cite{GS1}; in this paper we shall need the following consequence: 

\begin{cor}\label{cor:maxp}
Let $\psi(x,t)$ be a subsolution of the heat equation on $(M\setminus K)\times[0, T]$ for a compact set $K$. Suppose $\psi$ satisfies the above growth estimate. 
Set $\Gamma_T= (M\setminus K)\times\{0\}\cup \partial K\times [0,T]$. 
Then
\[\sup _{{\overline{(M\setminus K)}}\times[0, T]} \psi=\sup _{\Gamma_T} \quad \psi.\]
\end{cor}
\begin{proof}
   Set $v(x,t)=\psi(x,t)-\sup _{\Gamma_T} \psi $.
\[\phi(x,t)=
\begin{cases}
	\text{max}(v(x,t),0)&\text{on}\quad M\setminus K\times[0,T]\\
	0&\text{on}\quad K\times [0,T]
\end{cases}\]  Note that $\phi$ is a subsolution of heat equation because $v(x,t)=0$ on $\partial K\times [0,T]$. Furthermore, $\text{max}(v([x],t),0)$ is a weak subsolution of heat equation on $M\setminus K$. Therefore $\phi$ is a weak subsolution on $M\times [0,T]$. Now \[\int_{0}^{T}\int_{M}|\phi|^2e^{-cr^2([x])}d[x]dt<\infty,\quad\text{as}\quad\int_{0}^{T}\int_{M}|\psi|^2e^{-cr^2(x)}dxdt<\infty.\] Therefore by Proposition \ref{mpr}, $\phi=0$ and we are  done. 
\end{proof}

\vspace{.05in}

Finally, as in \cite{GS1} we shall need the Moser's Harnack inequality (see \cite{Moser}) for subsolutions of the heat equation:

\begin{prop}[\cite{WH}, Lemma 1.3, pg. 269]\label{mh}
Let  $0<R<1$ and $v\in C^{\infty}(B_R(x_0)\times[t_0-R^2,t_0])$ be a non-negative function on $(M,g)$ satisfying\[\left(\Delta_g-\frac{\partial}{\partial t}\right)v\geq-C v,\] for some $C>0$ on $ B_R(x_0)\times[t_0-R^2,t_0]$. Then there exists a constant $C_2>0$ such that\[v(x_0,t_0)\leq C_2R^{-(m+2)}\int_{t_0-R^2}^{t_0}\int_{B_R(x_0)}v(y,s)dvol_g(y) ds,\] where $m$ is the dimension of the ball $B_R(x_0)$.
\end{prop}

\vspace{.05in}

\section{Proof of Theorem \ref{thm:main}}

\subsection{Defining the initial map $u_0$}

Throughout, let $X$ denote the  hyperbolic  marked and bordered surface and let $\hat{X}$ denote its conformal completion.

\subsubsection{A pre-initial harmonic map $h$}

Recall that $\hat{\rho} = (\rho, \beta)$ is a non-degenerate type-preserving framed $\pslc$-representation. For a fixed ideal triangulation $T$ of $\sm$, we choose a signing $\epsilon:\mathbb{P} \to \{\pm 1\}$ such that $\hat{\rho}$ admits Fock-Goncharov coordinates with respect to the signed triangulation $(T, \epsilon)$ (see \cite[Theorem 9.1]{All-Bri}). Here, the signing modifies the framing as follows: if the monodromy around a puncture in $\mathbb{P}$ is loxodromic, then a framing has to necessarily assign it to one of its two fixed points; if $\epsilon(p)=-1$ we switch to the other fixed point. (See \cite[\S2.6]{GupWild} for a discussion.) 

We replace each (complex) Fock-Goncharov parameter by its modulus and let 
$\rho_0:\pi_1(\hat{X}) \to \pslr$ be the corresponding Fuchsian representation (see \S2.2), with framing $\beta_0$. We shall refer to the marked-and-bordered hyperbolic surface corresponding to $\rho_0$ as $X_0$; note that each crown or cusp or geodesic boundary end of $X$ corresponds to a crown or cusp or geodesic boundary end, respectively, of $X_0$.  The ideal triangulation $T$ determines a geodesic lamination on $X_0$, comprising finitely many geodesic lines between the marked-points and cusps, or spiralling onto the geodesic boundary components; switching the signing $\epsilon$ at a geodesic boundary component changes the direction of the spiralling. 

\vspace{.05in}

We start with the following $\rho_0$-equivariant harmonic map $h: \mathbb{H}^2 \to \mathbb{H}^3$ asymptotic to $\beta_0$ that exists by Theorem \ref{thm:all}. Namely, in the statement of the theorem, take the crowned hyperbolic surface $Y = X_0$,  lift the resulting harmonic map $F:\hat{X} \to Y \setminus \partial Y$ to a map between the universal covers $\widetilde{F}:\HH^2 \to \HH^2$ and then define $h = \iota \circ \widetilde{F}$ where $\iota:\HH^2 \to \HH^3$ is a totally geodesic embedding, say as the equatorial plane in the ball model of $\HH^3$. Here $\text{Hopf}(h)$ descends to a meromorphic quadratic differential on $\hat{X}$ having 
\begin{itemize}
    \item poles of order at most two at the punctures arising from $\mathbb{P}$, 
    \item poles of order $n\geq 3$ for each puncture arising from a crown end of $X$ having $(n-2)$ boundary cusps, and
    \item the prescribed principal parts at each pole. 
\end{itemize}  

Here, the prescribed principal part at each pole of higher order $n>2$ is compatible with the corresponding $(n-2)$-chain arising by considering the set $M_\infty$ of images under the framing $\beta$ of the lifts of marked points on the corresponding boundary component of $\sm$. (Here, ``compatible" is in the sense defined in  Definition \ref{defn:comp2}.) Such a principal part is also compatible with the straightening of $\mathcal{C}$ (see Definition \ref{defn:stn}), which by the definition of $X_0$ is precisely the lift of the corresponding crown end of $X_0$.

\vspace{.05in}

Note that the domain can be identified as the universal cover $\mathcal{H}$ of $\hat{X}$ (see Remark (ii) following Definition \ref{ccomp}) instead of the hyperbolic plane; throughout, we shall use a different conformal metric $g$ on the domain as defined below:

\begin{defn}[Metric $g$ on domain]\label{metg} The conformal metric $g$ on $\mathcal{H}$  is obtained by modifying the singular-flat metric induced by the Hopf differential  $q= \text{Hopf}(h)$. This modification shall be done at the level of the quotient surface $\hat{X}$, in particular on neighborhoods of each zero of $q$, and each puncture corresponding to pole of order one of $q$, such that the metric
\begin{itemize}
    \item[(a)] is the hyperbolic cusp metric in a (smaller) neighborhood of each order-one pole, 
    \item[(b)]is the flat $4q$-metric in the complement of the neighborhoods, 
    \item[(c)] is  $C^2$-smooth, and 
    \item[(d)] has non-positive Gaussian curvature. 
\end{itemize}

\end{defn}

\noindent We refer the reader to the Appendix for details of this metric interpolation.

\vspace{.05in}

In what follows, we shall modify the map $h$ to define the desired $\rho$-equivariant initial map $u_0: \HH^2 \to \HH^3$.

\subsubsection{Defining $u_0$ on lifts of crown ends}
We first describe how to modify $h$ and define $u_0$ on the lifts of the crown  ends of $\hat{X}$. 

For each such lift, this modification can be thought of as a $\mathbb{Z}$-equivariant version of the construction in \cite[\S3.1.3]{GS1}. In what follows, we shall try to follow the notation there, and we refer to it for details. 

Recall that the conformal metric on $\hat{X}$ near the $j$-th puncture is the metric induced by the Hopf differential of $u_0$, that has a pole of order $n_j$ at the puncture. Thus, this metric comprises $(n_j-2)$ horizontal and vertical half-planes,  arranged in a cyclic order around the puncture. (See \S 3 of \cite{GS1} for a description.) 
In the universal cover, each connected component of a lift of this neighborhood is a $\mathbb{Z}$-invariant chain of alternating horizontal and vertical half-planes, each intersecting with the next along a quarter-plane. 

 Each such chain $\mathcal{C}$ maps equivariantly to a chain $\mathcal{G}$ of hyperbolic geodesics that is a lift of the boundary geodesics of the corresponding crown end of $X_0$. Moreover, by Proposition \ref{prop:est}, we know that each horizontal half-plane  $H_i$ of $\mathcal{C}$ maps to a neighborhood of a geodesic, and each vertical half-plane $C_i$ to a neighborhood of a cusp bounded by two successive geodesics of $\mathcal{G}$, which we denote by $\tilde{C}_i$.  Here, for ease of notation, we have fixed an index $i \in \mathbb{Z}$; this shall also match with the notation employed in  \cite[\S3.1.3]{GS1}. 

 Recall that the framed representation $(\rho_0,\beta_0)$ is obtained by replacing each Fock-Goncharov parameter of $\hat{\rho}$ (which, for $\pslc$, is the shear-bend parameter of the corresponding edge of the ideal triangulation)  by its modulus. Reversing this, we can recover $\hat{\rho}$ by bending the totally-geodesic copy $P$ of $\HH^2$ preserved by the image of $\rho_0$, along the edges of $T$, which one can realize as a disjoint collection of geodesic lines on $P$. After bending, the original chain of geodesics $\mathcal{G}$ that lies on $P$ deforms to a chain of geodesics $\mathcal{G}^\prime$ in $\HH^3$ whose endpoints are precisely the image under $\beta$ of the (lifts of the) marked points corresponding to this end. 

If the planar cusped region $\tilde{C}_i$ in $P$  is between geodesic lines $\gamma_i$ and $\gamma_{i+1}$, let $\hat{C}_i$ be the planar cusped region lying in a totally-geodesic plane  in $\HH^3$ containing the two corresponding geodesics of $\mathcal{G}^\prime$, such that $\hat{C}_i$ is isometric to $\tilde{C}_i$. 

 To construct the map $u_0$, we first map each vertical half-plane $C_i$  to $\hat{C}_i$ in $\HH^3$ via a map that is the composition of the restriction of the harmonic map $h\vert_{C_i}$ with the isometry from  $\tilde{C}_i$ to $\hat{C}_i$. Note that post-composition with an isometry preserves harmonicity. 
 
 Next, we define $u_0$ on each horizontal half-plane as follows: note that $u_0$ is already defined on $C_i\cap H_i$ and $C_{i+1} \cap H_i$, and maps into two totally-geodesic planes $P_i,P_{i+1}$  that contain the cusps  $\hat{C}_i$ and  $\hat{C}_{i+1}$  respectively.  Here, $P_{i+1}$  can be obtained by applying an elliptic rotation  to $P_i$, where the rotation has an angle (say $\theta_0$) and with axis the geodesic line $\gamma_i$ that is common to the two successive cusps. Note that $C_i\cap H_i$ and $C_{i+1} \cap H_i$ defines two quarter-planes in $H_i$ (see \cite[Figure 4]{GS1}) with a half-strip, denoted by $D_i$, inbetween. On $D_i$, the map $u_0$ is defined to be an interpolation that rotates the original map $h$ by an angle that varies between $0$ on the left quarter-plane to $\theta_0$ on the right quarter-plane -- see equations (3.1) and (3.2) of \cite{GS1}. This construction can be done preserving the $\mathbb{Z}$-equivariance, such that the resulting map is $C^2$-smooth (see \cite[\S3.1.3]{GS1}).

\subsubsection{Completing the construction}
Next, we describe how to define the initial map $u_0$ in the lifts of the other types of ends of $\hat{X}$, namely a cusp and a cylinder-end (see Definition \ref{ccomp}) that arise from the interior marked points $\mathbb{P}$.  

\begin{itemize}
    \item[(i)] For a cusp end, each component of its lift to the universal cover $(\HH^2, g)$ is isometric to a horodisk in $\mathbb{H}^2$. The lift under the full Fuchsian group $\Gamma$ (the image of $\rho_0$) then comprises a countable collection of disjoint horodisks that can be thought of as embedded in the totally-geodesic plane $P$ preserved by $\Gamma$.  Since the representation $\rho$ is type-preserving, the image $\rho(\gamma)$ where $\gamma$ is a loop around the puncture is parabolic. Hence for each such horodisk lift in $P$, one can choose a corresponding isometric horodisk in a totally-geodesic plane  $\HH^3$  preserved by $\rho(\gamma)$.   The map $u_0$ on each horodisk in the lift of the cusp end is defined to be the restriction of the harmonic map $h$ post-composed by this isometric embedding. Once again, $u_0$ is harmonic on this lift. 

    \item[(ii)] For an end arising from the completion of a geodesic boundary component of $X$, each component of a lift to the universal cover is isometric to a Euclidean half-plane. Since $\rho$ is type-preserving, $\rho(\gamma)$ is a loxodromic element where $\gamma$ is the corresponding peripheral element, and has a geodesic axis.  Let $c_\theta$ be the  map that collapses the half-plane to its boundary line via the map  $(x,y) \mapsto x - y\tan\theta$, where $ae^{i\theta}$ is the principal part of the corresponding order-two pole. (See \cite[Definition 2.5]{DGT}.)  The map on each lift  is defined to be the map $c_\theta$  post-composed by an isometric embedding to the geodesic axis. Since the collapsing map is harmonic, so is the map $u_0$. 
    Note that our choice of collapsing map $c_\theta$ is such that its Hopf differential descends to a meromorphic quadratic differential with a pole of order two on $\mathbb{D}^\ast$ having the prescribed principal part $ae^{i\theta}$. 
\end{itemize}

Having defined $u_0$ on the lifts of neighborhoods $\{ U_i\}_{1\leq i\leq k + \lvert \mathbb{P} \rvert }$ of the  $k + \lvert \mathbb{P} \rvert $ punctures of $\hat{X}$, it remains to define it on the lift of their complement, i.e on the compact surface  $\bar{X} = \hat{X} \setminus \bigcup\limits_{i=1}^{k+\lvert \mathbb{P} \rvert}  U_i$ with boundary. Defining a $C^2$-smooth map from the universal cover of $\hat{X}$ to $\mathbb{H}^3$ that is $\rho$-equivariant is equivalent to constructing  a $C^2$-smooth section of the $\HH^3$-bundle over $\hat{X}$ associated to the representation $\rho$. So far, we have defined a section of this bundle over $\bigcup\limits_{i=1}^{k+\lvert \mathbb{P} \rvert}  U_i$; since the bundle is trivial over  the punctured surface $\hat{X}$, there is no obstruction in extending this to a $C^2$-smooth section of the entire surface. 
This defines the initial map $u_0:(\HH^2,g) \to \HH^3$; note that from our construction in the lifts of neighborhoods of the punctures, $u_0$ is asymptotic to the framing $\beta$. 

\vspace{.1in} 

Moreover, from our construction, we obtain the following estimates on the tension field and energy density of the initial map:

\begin{lem}\label{u0bd} 
For the $\rho$-equivariant map $u_0:(\HH^2,g) \to \HH^3$ constructed above, the norm of the tension field $\lvert \tau(u_0)\rvert \in L^\infty(\HH^2)$, and the energy density $e(u_0) \in L^\infty(\HH^2)$. Moreover,  $\lvert \tau(u_0)\rvert$ descends to a function $f: \hat{X} \to \mathbb{R}$ that 
\begin{itemize}
    \item is identically zero in neighborhoods of punctures  arising from the punctures $\mathbb{P}$, and 
    \item has an exponential decay in each neighborhood arising from a crown end, that is, $\lvert f(x) \rvert  = O(e^{-\alpha d(x, x_0)})$ where $x_0 \in \hat{X}$ is a choice of a basepoint and $\alpha>0$ is a constant.  
    
    (Here, the distance $d$ is with respect to the conformal metric $g$.) 
\end{itemize} 
\end{lem}

\begin{proof}
Since both the tension field and energy density of a map are unchanged when the map is post-composed by an isometry, the equivariance of $u_0$ ensures that both  descend to functions on $\hat{X}$. Thus, to prove the uniform-boundedness, it suffices to prove a bound on the ends (i.e. neighborhoods of the punctures) of $\hat{X}$ since their complement is compact. 

For the energy density $e(u_0)$, this bound arises from the following computations:
\begin{itemize}
    \item[(i)] For a cusp end, recall from \S3.1.3 that that $u_0$ restricts to an isometric embedding of a horodisk in $\HH^2$ into a horodisk on a  totally geodesic plane in $\HH^3$. Thinking of both the domain and the image as a subset of $\HH^2$ in the upper half-plane model, the map can be thought of as $f_1(x,y)=(x,0,y)$, and the energy density is \[e(f_1)=\lVert df_1\rVert^2=\sum_{i=1}^{3}y^2\left[\left(\frac{\partial f^i_1}{\partial x}\right)^2+\left(\frac{\partial f^i_1}{\partial y}\right)^2\right]\frac{1}{y^2}=2,\].

    \item[(ii)] For an end of $\hat{X}$  arising from a geodesic boundary component of $X$, recall from \S3.1.3 that $u_0$ restricts to a map that collapses an Euclidean half-plane to its boundary, post-composed by an isometric embedding to a geodesic line in $\HH^3$. Thinking of the target $\HH^3$ in the upper half-space model, the map is then $g(x,y) = (0,0, e^{x- c y})$ for some (fixed) constant $c$. This time, the energy density is:
    \[e(g)=\lVert dg\rVert^2=\sum_{i=1}^{3}\left[\left(\frac{\partial g^i}{\partial x}\right)^2+\left(\frac{\partial g^i}{\partial y}\right)^2\right]\frac{1}{e^{2x - 2cy}}=\left[e^{2(x-cy)}+c^2e^{2(x-cy)}\right]\frac{1}{e^{2(x - cy)}}=1+c^2.\] 

    \item[(iii] Finally, for the ends of $\hat{X}$ arising from the crown ends of $X$, recall from \S3.1.2 that the map $u_0$ on the horizontal (or vertical) half-planes of the $4q$-metric is an interpolation of harmonic maps to totally-geodesic copies of hyperbolic planes, and the energy density bound is a consequence of \cite[Lemma 3.2]{GS1}.
\end{itemize}

For the norm of the tension field: from our construction in \S3.1.3, the map $u_0$ is harmonic in the lifts of a neighborhood of a puncture of $\hat{X}$ arising from any point in $\mathbb{P}$ (i.e. corresponding to a cusp or boundary component of $X$), and hence the quotient function $f$ is identically zero there. For each crown end, recall once again that from \S3.1.2 the map on each half-plane of the $4q$-metric is an interpolation of harmonic maps, and by \cite[Lemma 3.3]{GS1} the norm of the tension field restricted to each such half-plane has an exponential decay (with respect to the distance from a chosen basepoint in the $4q$-metric). \end{proof}

\subsection{Uniform energy bound}

Using the initial map $u_0$ just defined, we shall now run the harmonic map heat flow \eqref{heat1};  the solution exists for all time by Theorem \ref{wang}. Here, the hypotheses of Theorem \ref{wang} holds since  the conformal metric $g$ on the domain surface has non-positive curvature everywhere, and Lemma \ref{u0bd} (in particular, the uniform-boundedness of the norm of the tension field) implies the finiteness of $b(x,t)$. 

From our construction, the initial map $u_0:\mathbb{H}^2 \to \mathbb{H}^3$ is $\rho$-equivariant, where $\rho:\pi_1(\hat{X})\to \pslc$, and the action is by deck-translations in the domain (which is the universal cover of $\hat{X}$), and by the image of $\rho$ on the target.

\begin{lem}\label{lem:equiv}
    For each $t>0$, the map $u_t:\mathbb{H}^2 \to \HH^3$ is $\rho$-equivariant.
\end{lem}
\begin{proof}
    This is an immediate consequence of the uniqueness of the flow, that follows from Theorem \ref{wang}. Let $\gamma \in \pi_1(\hat{X})$, and consider the  harmonic map heat flow  with the initial map $u_0 \circ \gamma$:
        \begin{equation}\label{uheat}
        \begin{split}
\frac{\partial u}{\partial t}&=\tau(u(x,t))\\
u(x,0)&=u_0(\gamma.x)  
\end{split}
\end{equation} 
    
    Since $\gamma$ acts by isometries on the domain, it leaves the tension field invariant (see, for instance \cite[pg. 162, Exercise 4.4]{Nishikawa02}). In particular, the tension field of $u_0 \circ \gamma$ is bounded, and the existence (and uniqueness) of \eqref{uheat} also holds for this initial map. Define\[v(x,t)=A(u(x,t)),\]  where $A=\rho(\gamma)$. Then \[\frac{\partial v}{\partial t}=dA_{u_t(x)}\left(\frac{\partial u}{\partial t}\right)=dA_{u_t(x)}(\tau(u(x,t))=\tau(A(u(x,t)))=\tau(v(x,t)),\] where $v(x,0)=\rho(\gamma(u_0(x)))=u_0(\gamma.x)$. Therefore, $v(x,t)$ and $u(\gamma. x,t)$ are solutions of \eqref{uheat} with the same initial condition. By uniqueness $u_t(\gamma.x)=\rho(\gamma)(u_t(x))$ for all $x \in \mathbb{H}^2$.
\end{proof}

\vspace{.05in}

\noindent We also show: 

\begin{lem}\label{ebound} 
There is uniform bound on the energy density $e(u_t)$, that is, there exists a constant $D>0$ such that $e(u_t(x))\leq D$ for all   $(x,t) \in \mathbb{H}^2 \times [0,\infty)$.
\end{lem}
\begin{proof}
Recall that the universal cover of the conformal completion $\hat{X}$ is $\mathbb{H}^2$ equipped with the conformal metric g, that we denote by $\mathcal{H}$. 
    Fix $x_0\in\mathcal{H}$  and let $B_{R}(x_0)$ denote the (relatively compact) ball of radius $0<R<1$ centered at $x_0$, containing $K$. Since $Ric_g \geq -Ag $, for some $A>0$,   using  Weitzenbock Formula for the energy density function $e(u_t)$ (see \cite[Proposition 4.2 (1) and Corollary 4.3 (1)]{Nishikawa02}), we have 
    \begin{equation}\label{weitcor}
    \left(\Delta_g-\frac{\partial}{\partial t}\right)(e(u_t))\geq -Ae(u_t)
    \end{equation}
    where $e(u_t)$ is considered a function over $B_R(x_0) \times [0, \infty)$.
    
Applying  Moser's Harnack inequality (Proposition \ref{mh}) to $(x_0,t_0)\in B_R(x_0)\times [1,T)$, we obtain
\begin{align*}
e(u_{t_0})(x_0)&\leq CR^{-4}\int_{t_0-R^2}^{t_0}\int_{B_R(x_0)}e(u_t)dvol_g dt\\
&=CR^{-4}\int_{t_0-R^2}^{t_0}E_R(u_t)dt,
\end{align*}
where $E_R(u_t)=\int_{B_R(x_0)}e(u_t)dvol_g$. Since the total energy on $B_R(x_0)$ is a decreasing function of $t$, we obtain \[e(u_{t_0})(x_0)\leq CR^{-2}E_R(u_0).\] 

Let $D_0>0$ be the uniform bound on $e(u_0)$ (from Lemma \ref{u0bd}), we know that $E_R(u_0) \leq D_0 \text{Area}(B_R(x_0)) \leq CR^2$ for some (universal) constant $C>0$, since the area growth of disks in the conformal metric $g$ on the domain  is quadratic in the radius. Putting this in the inequality above, we conclude that  $e(u_t)$ is uniformly bounded on $B_R(x_0)\times (1,T)$, where the bound is independent of $x_0,R$ and $T$. 

It remains to bound the energy density on $B_R(x_0)\times [0,1]$. For this, we write $a(x,t)=e^{-t}e(u_t)$ and using \eqref{weitcor} we obtain \[\left(\Delta_g-\frac{\partial}{\partial t}\right)(a(x,t))\geq 0.\] 

Then by the Parabolic Maximum Principle for manifolds with boundary (Proposition \ref{boundary}) applied to $M = \overline{B_R(x_0)}$, we conclude that for each $(x,t) \in B_R(x_0) \times [0,1]$ we have 
\[a(x,t)\leq \underset{\overline{B_{R_0}(x)}\times \{0\}\cup {\partial B_{R_0}(x)\times [0,1]} }{sup}a(x,t)\] 
which implies, from the definition of $a(x,t)$, that
\begin{equation}\label{ino}
e(u_{t})(x)\leq e(u_0)(x)e\leq D_0e
\end{equation} 
where as before, this bound is independent of $x_0$ and $R$. 

This completes the proof of the uniform bound of $e(u_t)$.
\end{proof}

\noindent\textit{Remark.} Note that we also have a uniform bound on the norm of the tension field $\lvert \tau(u_t) \rvert$ that is independent of $t$, as a consequence of the parabolic maximum principle (Proposition \ref{mpr}), since $\lvert \tau(u_0) \rvert$ is uniformly bounded by Lemma \ref{u0bd} and $\lvert \tau(u_t) \rvert$ is a weak subsolution of the heat equation (see \cite[Lemma 2.1]{MengWang}). 

\subsection{Proof of a ``one-point convergence" result}

In this subsection, we prove a ``one-point convergence" result that is a crucial first step towards proving the convergence of harmonic map heat flow. As discussed in \S2.4, this convergence is not true in general, even if the flow exists for all time; here we shall use the non-degeneracy of the representation.

We start with the following consequence of the uniform bound on the energy density $e(u_t)$ (which recall is the norm of the total derivative of $u_t$):

\begin{prop}\label{bound}
For any  $x\in\mathbb{H}^2$, $d(u_t(\gamma.x),u_t(x))\leq B$, where $B$ is independent of $x$ and $t$.
\end{prop}
\begin{proof}
Let $\sigma:[0,1]\rightarrow\mathbb{H}^2$ be a curve joining $x$ and $\gamma.x$. 
Now 
\begin{align*}
d(u_t(\gamma.x),u_t(x))&\leq \text{length}(u_t(\sigma)) = \int_{0}^{1}	(e(u_t(\sigma(s)))^{1/2}ds 
\leq \int_{0}^{1}	(e(u_t(\sigma(s)))ds
\leq D,
\end{align*}
where the second inequality is by Cauchy-Schwarz, and the third is from the uniform bound on the energy density obtained in Lemma \ref{ebound}.
\end{proof}

\vspace{.05in}

By Lemma \ref{eqivuse}, there exists a $\gamma\in \Pi$ such that $\rho(\gamma)$ is a semi-simple element with an axis $L$. Choose $x_0\in \mathbb{H}^2$ such that $x_0$ and $\gamma.x_0$ lie on  opposite sides of a fundamental polygon. 

\vspace{.05in}

\noindent We first observe:

\begin{lem} \label{mt}
The images $u_t(x_0)$ must lie in some bounded neighborhood (independent of $t$) of the axis of $\rho(\gamma)$. 
\end{lem}
\begin{proof}
It is enough to show that $u_t(x_0)$ must lie in a uniformly bounded neighborhood of the axis of $\rho(\gamma)$.
Assume not, i.e., let $\{x_n\}_{n\geq 1}$ be a sequence in $\mathbb{H}^2$, such that $y_n := u_t(x_n)$ escapes any bounded neighborhood of the axis of $\rho(\gamma)$. 
Using a basic fact from hyperbolic geometry, we can say that  $d_{\mathbb{H}^3}(\rho(\gamma)(y_n),y_n)\rightarrow\infty$ as $n\rightarrow\infty$. 

Applying $\rho$-equivariance of $u_t$ we have \[d_{\mathbb{H}^3}(u_t(\gamma.x_n),u_t(x_n))=d_{\mathbb{H}^3}(\rho(\gamma)(u_t(x_n)),u_t(x_n))=d_{\mathbb{H}^3}(\rho(\gamma)(y_n),y_n)\rightarrow\infty,\]
as $n\to \infty$, contradicting Proposition \ref{bound}.
\end{proof}

\noindent As a consequence of the above lemma, we obtain:

\begin{lem}\label{newm}
The images $u_t(x_0)$ lie in a compact set (independent of $t$)  in $\mathbb{H}^3$.
\end{lem}
\begin{proof}
We use the fact that if $L$ and $L'$ are two geodesic lines in $\mathbb{H}^3$ with distinct endpoints, then for any $A, B \in \mathbb{R}^+$, $N(L, A)\cap N(L', B)$ is a compact set in $\mathbb{H}^3$, where $N(L, A) = \{ x\in \HH^3\ \vert\ d(x, L) \leq A \}$ is a closed $A$-neighborhood of $L$. Using Lemma \ref{eqivuse}, we have two elements $\gamma,\gamma'$ in $\pi_1(X)$ such that $\rho(\gamma),\rho(\gamma')$ are two different elements with two different axes. Finally, by applying Lemma \ref{mt} to each of the two axes, $u_t(x_0)$ lies in the intersection of neighborhoods of these two axes, which is compact.
\end{proof}
\begin{cor}\label{opc}
 There exists a point $p$ in $\mathbb{H}^2$ and a sequence $\{t_n\} \subset[0,\infty)$ such that $u_{t_n}(p)$ converges.   
\end{cor}

\subsection{Convergence of the flow}

In the rest of this subsection, we shall prove that the maps $u_t$ in fact converge, as $t\to \infty$, to our desired harmonic map.

\subsubsection{Subsequential convergence}
The one-point convergence proved in the previous section is  enough to prove the following:

\begin{lem}\label{lem:subseq}
There exists a sequence $\{t_n\}\subset[0,\infty)$ such that $u_{t_n}$ converges to a $\rho$-equivariance harmonic map $u_{\infty}:\HH^2 \to \HH^3$ uniformly on compact subsets of $\mathbb{H}^2$ together with its derivatives.
\end{lem}
\begin{proof}
Take a compact subset $K$ of $\mathbb{H}^2$. The discussion in the previous section implies that there exists a point $p$ in $\mathbb{H}^2$ such $u_{t_n}(p)$ converges as $t_n\rightarrow\infty$.  By Lemma \ref{ebound} the energy density of $u$ is uniformly bounded on $K\times [0,\infty)$; this implies that the sequence $u_n(x):=u(x,t_n)$ is equicontinuous and together with Corollary \ref{opc}, shows that the sequence is point-wise bounded. Therefore by Arzela-Ascoli's theorem, after passing to a subsequence, $u_n \to u_{\infty}$ uniformly on $K$. Since the compact set $K$ is arbitrary, this show that $u_n \to u_{\infty}$ uniformly on compact subsets. Applying Schauder estimate (see Appendix $A.2(e)$ of \cite{Nishikawa02}) we obtain the convergence of space derivatives of $u_n$ up to second order to the corresponding derivatives of $u_{\infty}$. Therefore $|\tau(u_n)|$ converges to $|\tau(u_{\infty})|$ uniformly on compact subsets as $n\to \infty$. Now we shall recall the following argument of Eells-Sampson \cite[pg. 136]{EellsSampson} to show that the restriction $u_\infty\vert_U$ on any relatively compact open set $U$ is harmonic:

First, the  first variation of energy yields\[F(t):=\frac{d}{dt}E_U(u_t)=-\int_{U}\left\langle\frac{\partial u_t}{\partial t},\tau(u_t)\right\rangle dvol_g=-\int_{U}|\tau(u_t)|^2dvol_g.\] where $E_U(u_t)$ is the total energy of the restriction $u_t\vert_U$. 

Recall that under the harmonic map heat flow,  this total energy $E_U(u_t)$  decreases; applying second variation formula together with the fact that the all sectional curvatures of the target manifold $\HH^3$ are non-positive, the argument of Eells and Sampson \cite[pg. 136]{EellsSampson} shows the convexity  \[\frac{d^2}{dt^2}E_U(u_t)\geq 0,\] and consequently, we have \[0=\underset{n\rightarrow\infty}{\lim}F(t_n)=-\int_{U}\underset{n\rightarrow\infty}{\lim}|\tau(u_n)|^2dvol_g=-\int_{U}|\tau(u_{\infty})|^2dvol_g,\]  yielding $\tau(u_{\infty})=0$ on $U$.  Since this holds for all relatively compact open sets $U$ in $\HH^2$, the limiting map  $u_{\infty}$ is harmonic. The $\rho$-equivariance of $u_\infty$ is an immediate consequence of Lemma \ref{lem:equiv}.
\end{proof}

\subsubsection{Bounded distance from initial map}

The proof of the following will use the exponential decay of the tension field (see Lemma \ref{u0bd}):

\begin{prop}\label{pmp}
The distance function $d(u_t(x),u_0(x)) \leq M$ for a uniform constant $M>0$ (independent of $x$ and $t$). 
\end{prop}
\begin{proof} 

By the $\rho$-equivariance of $u_t$ (Lemma \ref{lem:equiv}), the distance function  $\psi(x,t) =  d(u_t(x), u_0(x))$ descends to a function $\widehat{\psi}: \hat{X} \times [0,\infty) \to \mathbb{R}$. We wish to prove that $\widehat{\psi}$ is uniformly bounded.

In the case that $X$ does not have crown ends, i.e. $\mathbb{S}$ does not have boundary components (and $\mathbb{M} = \mathbb{P}$), then  recall that the norm of the tension field descends to a function $\hat{\tau}$ that is supported on a compact set $K \subset \hat{X}$, and hence by \cite[Lemma 2.1]{GS1} the following holds:
\begin{equation}
\left(\Delta_{\bar{g}}-\frac{\partial}{\partial t}\right)\widehat{\psi}([x],t)=\left(\Delta_{g}-\frac{\partial}{\partial t}\right)\psi(x,t)\geq 0
\end{equation}
for all $[x]\in \hat{X}\setminus K$.  Here $\bar{g}$ is the metric on $\hat{X}$ that lifts to $g$ in the universal cover (\textit{c.f.} Definition\ref{metg}).

Note that we can assume that the point $p$ of Corollary \ref{opc} lies in $K$. Since the uniform bound on energy density (Lemma \ref{ebound}) implies that the restriction of $u_t$ on $K$ is uniformly Lipschitz, we obtain a uniform bound
\begin{equation}\label{use0}
\widehat{\psi}([x], t) \leq M 
\end{equation}
for all $[x]\in K$, where $M>0$ is a  uniform constant (independent of $t$). 

We can then apply the maximum principle (Corollary \ref{cor:maxp}) to conclude that $\widehat{\psi}$ is uniformly bounded.

\vspace{.05in}

In the case that $X$ has crown ends, or equivalently, $\mathbb{S}$ has boundary components, the norm of the tension field $\psi$ is not compactly supported; however, by Lemma \ref{u0bd} the it has exponential decay towards the punctures of $\hat{X}$ corresponding to the crown ends. 

We can choose a neighborhood $U$ of such a  puncture such that $U$ is homeomorphic to $\mathbb{C} \setminus \overline{\mathbb{D}}$ (with the puncture at infinity) and the metric $\bar{g}$ on $U$ is conformally flat, i.e. of the form $g(z) \lvert dz^2 \rvert$. Note that by Definition \ref{metg}, the conformal factor  $g(z)$ on $U$  is the norm of the Hopf differential, that has at most polynomial growth (as $\lvert z \rvert \to \infty$).

It follows that 
\begin{equation}\label{use} 
   \Delta_{\bar{g}}\widehat{\psi}(x,t) =  \Delta_{g}d(u_t(x),u_0(x))\geq-(\lvert \tau(u_t)\rvert+|\tau(u_0)|)  \geq -2\lvert \tau_0 \rvert
\end{equation} where the first inequality follows from the computations in \cite{SchoenYau79} (see also \cite[(3.4)]{JiayuLi93} or \cite[(2.1)]{DingWang}), and the  second inequality follows from the remark following Lemma \ref{ebound}.

Note that if, instead, we had that $\Delta_{\bar{g}} \widehat{\psi} \geq 0$, then it would be a subsolution of the heat equation, and the uniform boundedness of $\psi$  that we are trying to prove would be a consequence of the maximum principle (Corollary \ref{cor:maxp}) as before. Instead, we shall apply the maximum principle to the function $w([x], t) = \widehat{\psi}([x],t)+F([x])$ where $F:U \to \mathbb{R}$ is the   bounded subharmonic function defined by $F(r,\theta) = \frac{1}{r^k}$, where we  choose $k\in\mathbb{N}$ suitably such that $\Delta_{\bar{g}}F=O(\frac{1}{r^s})$ for some $s>0$, i.e. has \textit{polynomial} decay in $r$. Now we have,

\begin{align}\label{use2}
    \left(\Delta_{\bar{g}}-\frac{\partial}{\partial t}\right){w}=& \left(\Delta_{\bar{g}}-\frac{\partial}{\partial t}\right){(\widehat{\psi}+F)} \\ \nonumber
=&\left(\Delta_{g}-\frac{\partial}{\partial t}\right)\psi+\Delta_{\bar{g}}F \geq  -2\lvert {\tau}(u_0)\rvert - \frac{\partial\psi}{\partial t}+\Delta_{\bar{g}}F
\end{align}
where the inequality follows from \eqref{use}.

Setting $h(x,t) = (u_t(x), u_0(x)) \in \HH^3 \times \HH^3$, we can apply the chain rule to write
\begin{equation}
\frac{\partial\psi}{\partial t} = \langle \nabla d, \frac{\partial h}{\partial t} \rangle = \langle \nabla d, (\tau(u_t), 0) \rangle \leq \lvert \nabla d \rvert  \lvert \tau(u_t) \rvert \leq \sqrt{2}  \lvert \tau(u_0) \rvert 
\end{equation}
where the first inequality is by Cauchy-Schwarz and the second inequality is from \cite[(1.1)]{SchoenYau79} and the remark following Lemma \ref{ebound}.

Using this in \eqref{use2}, we obtain
\begin{equation}\label{use3} 
     \left(\Delta_{\bar{g}}-\frac{\partial}{\partial t}\right){w} \geq - (2 + \sqrt{2}) \lvert \tau(u_0) \rvert + \Delta_{\bar{g}}F
\end{equation}

Then using the fact that $\lvert \tau(u_0)\rvert$ decays exponentially in the crown end (Lemma \ref{u0bd}), one can choose a compact set $K$  such that the right hand side of \eqref{use3} is non-negative, that is,  \[\left(\Delta_{\bar{g}}-\frac{\partial}{\partial t}\right){w}([x],t) \geq 0,\] for all $[x] \in \hat{X} \setminus K$ and all $t>0$. 
Applying the maximum principle (Corollary \ref{cor:maxp}) as before, with the parabolic boundary condition guaranteed by \eqref{use0}, we obtain a constant $M^\prime >0$ such that $w([x],t)\leq M^\prime$ for all $[x] \in \left(\hat{X} \setminus K \right) \times \mathbb{R}^+$. Since $F$ is bounded, we obtain a uniform bound on $\psi(x,t)$ as desired. 
\end{proof}

\textit{Remark.} One can show that $\hat{\psi}(\cdot, t)$, i.e.\ the norm of the tension field of $u_t$ has an exponential decay towards the punctures of $\hat{X}$ following the same argument as \cite[Lemma 3.10]{GS1}. This would imply that for any fixed $t$, the distance function $d(u_t(x),u_0(x))\rightarrow0$ as $x$ diverges on $\hat{X}$. 

\vspace{.1in} 

\noindent As an immediate corollary, we now conclude: 

\begin{cor}\label{cor:asymp}
The harmonic map $u_{\infty}$ satisfies the uniform distance bound
\begin{equation}\label{db}
d(u_\infty(x),u_0(x)) \leq M
\end{equation}
for all $x\in \HH^2$, and  is asymptotic to the given framing $\beta$. Moreover, the Hopf differential $\text{Hopf}(u_\infty)$ descends to a meromorphic quadratic differential on $\hat{X}$ having poles of order at most two at the punctures arising from $\mathbb{P}$, and a pole of order $(m+2)$ at every puncture arising from a crown end of $X$ having  $m$ boundary cusps. 
\end{cor}
\begin{proof}
The distance bound \eqref{db} follows from Proposition \ref{pmp} since $u_t\to u_\infty$ uniformly on compact sets.  By construction of the initial map $u_0$, it is asymptotic to $\beta$. Then, if $\{x_n\}_{n\geq 1}$ is a sequence in $\mathcal{H} \cong \mathbb{H}^2$ such that $x_n\rightarrow x$ where $x\in F_{\infty} \subset \partial_\infty \widetilde{X}$, we know that $u_0(x_n)\rightarrow\beta(x)$. Recall that if two diverging  sequences  in $\mathbb{H}^3$ are at a bounded distance from each other, then they converge to the same point in the ideal boundary.  Thus, by \eqref{db} we conclude that $u_{\infty}$ is asymptotic to the given framing $\beta$. 

The distance bound  also implies that on each crown end $U \cong \mathbb{D}^\ast$ of $\hat{X}$, the restriction of the total energy $E(u_\infty\vert_U)$ descends to a function  descends to a function that has at most polynomial growth (in $1/r$), and hence the Hopf differential $\text{Hopf}(u_\infty)$ descends to a meromorphic quadratic differential on $U$ having an finite order. By Lemma \ref{map2chain}, this finite order must equal $(m+2)$ where $m$ is the number of marked points on the corresponding boundary of $\sm$. 

Similarly, for any cylinder end $U'$ of $\hat{X}$, the Hopf differential $\text{Hopf}(u_\infty)$ descends to a meromorphic quadratic differential on $U$ having a pole of order $2$ by Lemma \ref{map2}. 

Finally, by the computation in Lemma \ref{map3}, at a cusp end $U''$ of $\hat{X}$, the Hopf differential  $\text{Hopf}(u_\infty)$ descends to $U''$ with a pole of order one at the puncture. Here, note that from the type-preserving assumption on $\rho$, we know that the corresponding peripheral element of $\pslc$ is parabolic. \end{proof}

\subsubsection{Endgame}

We start with the following uniqueness result, the proof of which is identical to the argument in \cite{Sagman}:

\begin{lem}\label{uniqueness}
The map $u_{\infty}$ is the unique $\rho$-equivariant harmonic map which is at a bounded distance from $u_0$.
\end{lem}
\begin{proof}
Let $u_1,u_2$ be two $\rho$-equivariant harmonic maps at a bounded distance from $u_0$. Define $\Psi:\HH^2 \to \mathbb{R}_{\geq 0}$ by $\Psi(x)=d(u_1(x),u_2(x))$. Then by Lemma \ref{sy} (see the remark following it), we know that $\Psi$ is a bounded subharmonic function on $\mathbb{H}^2$. Moreover, by  the equivariance property, $\Psi$  descends to a bounded subharmonic function $\psi: \hat{X} \to \mathbb{R}_{\geq 0}$. As the punctured Riemann surface $\hat{X}$ is parabolic in the potential theoretic sense, this function is constant. Then the argument in \cite[Chapter 11.2]{SchoenYau97} (see also \cite[Lemma 3.11]{Sagman}) implies that either $u_1=u_2$ or $u_1$ and $u_2$ have an image in a geodesic and differ by a translation along that geodesic. However, the latter case can only occur if the image of $\rho$ stabilizes a geodesic; this is ruled out by the fact that the framed representation $\hat{\rho}$ is non-degenerate which implies that $\rho$ does not have a global fixed point in the ideal boundary of $\mathbb{H}^3$. Hence, $u_{\infty}$ is unique.
\end{proof}

As a consequence, we can now prove: 
\begin{lem}\label{lem:conv}
The harmonic map heat flow $u_t$ converges to $u_\infty$.
\end{lem}
\begin{proof}
Let  $\psi(x,t)=d(u_{\infty}(x),u_t(x))$ where $u_\infty$ is the harmonic map obtained as a  subsequential limit earlier in this section. Since $u_\infty$ is a stationary solution of the harmonic map heat flow, we obtain from Lemma \ref{sy} that  $\psi$ is a subsolution of the heat equation. From Corollary \ref{cor:asymp} we also have the uniform bound $\psi(x,0)\leq M$ for all $x\in \HH^2$. Note that \[\psi(x,t)\leq d(u_{\infty}(x),u_0(x))+d(u_0(x),u_t(x))\leq  M+ \int\limits_{0}^t \left\vert \frac{\partial u_t}{\partial t}  \right\vert ds = M+  \int\limits_{0}^t \lvert \tau (u_s) \rvert ds \leq M+Dt,\] for some $C,D>0$, where last inequality uses the uniform-boundedness of the norm of the tension field (see the remark following the proof of Lemma \ref{ebound}).

Thus, for any $T>0$, \[\int_{0}^{T}\int_{\mathbb{C}} e^{-cr^2(x)}\psi^2(x,t)dxdt\leq\int_{0}^{T}\int_{\mathbb{C}} e^{-cr^2(x)}(C+Dt)^2dxdt<\infty\]for  some $c>0$. 

We can then apply the prabolic maximum principle (Proposition \ref{mpr}) to conclude that the uniform distance bound $d(u_{\infty}(x),u_t(x))\leq M$ holds for all $t$, not just along a subsequence. 
Thus any subsequential limit along the flow must be a uniformly bounded distance from $u_0$, and hence by Lemma \ref{uniqueness}, equal $u_\infty$. Hence the harmonic map heat flow converges, and the limiting harmonic map is $u_\infty$. 	
\end{proof}

\subsection{Uniqueness and the proof of Theorem \ref{thm:main}}

The preceding sections have completed the existence statement in Theorem \ref{thm:main}. 
The uniqueness statement of Theorem \ref{thm:main} follows exactly as in \cite{GS1}; in what follows we shall refer to the relevant Lemmas in that paper. 

First, recall that in the construction of the initial map in \S3.1, we started with a $\rho_0$-equivariant harmonic map $h:\mathbb{H}^2 \to \HH^2$; by Theorem \ref{thm:all}, we can prescribe the principal parts at the punctures, as long as they are compatible with the corresponding ends of $X_0= \mathbb{H}^2/\Gamma_0$ where $\Gamma_0$ is the image of the Fuchsian representation $\rho_0$, where compatible is in the sense defined in Definition \ref{defn:comp}.

\vspace{.05in}

Let $\mathbb{P}$ be the prescribed principal part at a puncture of $\hat{X}$, that is the principal part of $\text{Hopf}(h)$ at that puncture, with respect to a coordinate neighborhood $U \cong \mathbb{D}^\ast$. We shall show that it is also the principal part of $\text{Hopf}(u)$ where $u= u_\infty$ is the harmonic map obtained above (see Lemma \ref{uniqueness}). 
For this, it suffices to consider the restrictions of the harmonic maps to the universal cover $\widetilde{U}$ which is conformally the upper half-plane. Note that the restrictions $h\vert_{\widetilde{U}}$ and $h\vert_{\widetilde{U}}$ are $\mathbb{Z}$-equivariant asymptotic to chains of geodesics (see Definition \ref{defn:chain}) that we denote by $C_0$ and $C$. 
\vspace{.05in}

We need the following version of \cite[Lemma 4.2]{GS1}. Here, we consider $\mathbb{Z}$-equivariant harmonic maps to chains such that the quotient Hopf differential has a pole of order at least $3$ (see Lemma \ref{map2chain}). 

\begin{lem}\label{lem:42}
    Suppose $U \cong \mathbb{D}^\ast$ and  $h,u:\widetilde{U} \to \HH^3$ are $\mathbb{Z}$-equivariant harmonic maps that are asymptotic to chains $C_0$ and $C$ respectively, and are normalized such that the images along a fixed geodesic towards $\infty$ in $\widetilde{U}$ is asymptotic to the same ideal point. Let $\mathbb{P}$ and $\mathbb{P}^\prime$ be the principal parts of the quotient Hopf differentials defined on $U \cong \mathbb{D}^\ast$, that have poles of order $n\geq 3$. Then there is a uniform distance bound 
\begin{equation}\label{hub}
    \lvert d(h(x), p_0) - d(u(x), p_0) \rvert < D \text{ for all }x\in F_0
\end{equation}
where $F_0$ is a choice of fundamental domain for the $\mathbb{Z}$-action on $\widetilde{U}$, if and only $\mathbb{P} =\mathbb{P}^\prime$. (Here, $p_0$ is a  basepoint in $\HH^3$, and the constant $D>0$ is independent of $x$.) 
\end{lem}

\begin{proof}
The proof is essentially identical to the proof of \cite[Proposition 3.9]{GupWild} (see also \cite[Lemma 4.2]{GS1}), and as in that paper, relies on the estimates in \cite[Lemmas 3.7,3.8]{GupWild}.  Note that we can always choose such a normalization of $h,u$ by pre-composing with an appropriate conformal automorphism of the domain (namely, a lift of a rotation of $U$). 

Let $q_1,q_2$ be the meromorphic quadratic differentials on $U\cong \mathbb{D}^\ast$ obtained as a quotient of the $\mathbb{Z}$-invariant Hopf differentials $\text{Hopf}(h), \text{Hopf}(u)$ respectively on $\widetilde{U}$.

In one direction, assume that the two principal parts $\mathbb{P}, \mathbb{P}^\prime$ of $q_1, q_2$ are equal. Then, following  the proof of \cite[Proposition 3.9]{GupWild}, we consider an ``polygonal exhaustion" $\{P_j\}_{j\geq 1}$ of $U$ for the $q_1$-metric, that is,  an exhaustion by compact sets bounded by polygons $P_j$, each comprising edges that are alternately segments of leaves of the horizontal and vertical foliations of $q_1$ (see \cite[Definition 2.27]{GupWild}). Note that each $P_j$ has exactly $m = n-2$ such horizontal and $m$ vertical edges. 

By the equality of principal parts and \cite[Lemma 3.7]{GupWild}, the lengths of the  sides of each such polygon $P_j$  in both the $q_1$-metric and the $q_2$-metric differ by a uniformly bounded constant.  (We can choose this exhaustion such that the lengths of the sides of $G_k$ are all $L_k$ up to a uniform additive error, and  $L_k \to \infty$ as $k\to \infty$). 

Consider a polygon $P_j$ in $U$; its lift $\tilde{P}_j $ to fundamental domain $F_0 \subset \widetilde{U}$ is a finite chain of $m$ horizontal and vertical edges. 
There are $m$ corresponding ideal points in  fundamental domains of the $m$-chains $C_0$ and $C$, respectively.

By the normalization, we can assume that $h$ (respectively $u$) take the $i$-th vertical side (where $1\leq i\leq m$) of $\tilde{P}_j$ into the $i$-th ``cusp" in the fundamental domain of the $\mathbb{Z}$-action on  $C$. Moreover, by Proposition \ref{prop:est}, and an argument as in the Claim in the proof of \cite[Proposition 3.9]{GupWild}, the $i$-th side will be a distance $L_k + O(1)$ into the $i$-th cusp, for both maps. Choosing a basepoint $p_0 \in \HH^3$, \eqref{hub} follows, namely the difference of distances from $p_0$ is uniformly bounded. 

Conversely, if $\mathbb{P} \neq \mathbb{P}^\prime$, then by \cite[Lemma 3.8]{GupWild} there is a sequence of points $z_i \in U$ tending to the puncture, such that the horizontal distances $H_i$ and $H_i^\prime$ with respect to the $q_1$-metric and the $q_2$-metric  respectively, of $z_i$ from a fixed basepoint $z_0\in U$, have an unbounded difference, i.e. $\lvert H_i - H_i^\prime \rvert \to \infty$ as $i\to \infty$. By passing to a subsequence, one can assume that this sequence of points lie in the same vertical half-plane in both metrics. Then the images of the lifts  $\tilde{z_i} \subset F_0$ are mapped by $h$ towards a fixed cusp of $C_0$,  and their images under $u$ tend to a fixed cusp  of $C$.  By the previous estimates, the distances of these images from a fixed point $p_0\in \HH^3$ are of the order of the horizontal distances $H_i$ and $H_i^\prime$ respectively, and thus have an unbounded difference, $\lvert d(h(z_i), p_0) - d(u(z_i), p_0)\rvert \to \infty$ as $i\to \infty$ contradicting \eqref{hub}. 
\end{proof}

\noindent We can now formally complete the proof of our main Theorem: 

\begin{proof}[Proof of Theorem \ref{thm:main}]
As described in \S3.1, given the non-degenerate framed representation $\hat{\rho} = (\rho,\beta)$, and a choice of ideal triangulation $T$ of $\sm$, there is a (framed) Fuchsian representation $(\rho_0, \beta_0)$ that corresponds to a crowned hyperbolic surface $X_0$.   This Fuchsian representation is obtained by replacing the Fock-Goncharov coordinates of $\hat{\rho}$ by their absolute values.  By Theorem \ref{thm:all}, fixing a choice of compatible principal parts at the punctures, there is a  (unique) $\rho_0$-equivariant harmonic map $h:\HH^2 \to \HH^2$ where the domain is the universal cover of the punctured Riemann surface $\hat{X}$, and the target is identified with a totally geodesic plane $P$ in $\HH^3$. This map is in fact a diffeomorphism from the universal cover (the domain) to a convex subset of $P$ corresponding to the lift of $X_0$. 

From the geometric interpretation  of the Fock-Goncharov coordinates (see, for example, \cite[\S2.2-3]{GuptaSu23} for a discussion), the framing $\beta$ is obtained by ``bending" or ``pleating" the totally geodesic plane $P$ along the lifts of the edges of the ideal triangulation $T$.  After such a bending $P$ transforms into a piecewise totally-geodesic ``pleated plane" that is invariant under the image of $\rho$, and asymptotic to $\beta$ -- we denote the harmonic map $h$ post-composed by this bending deformation by $\Xi:\HH^2 \to \HH^3$. Note that since each ``bend" is an elliptic isometry, it moves points a bounded distance (where the bound only depends on the distance from its axis). This ensures that the distance function $\lvert d(h(x), p_0) - d(\Xi(x), p_0)\rvert$ is uniformly bounded for all $x\in F_0$, where $F_0$ is the fundamental domain of the $\Pi$-action on $\HH^2$.  From our construction of the initial map $u_0$ in \S3.1, it is a uniform distance away from this pleated plane $\Xi$; thus, by the triangle inequality, the distance function  $\lvert d(h(x), p_0) - d(u_0(x), p_0)\rvert$ is also uniformly bounded for $x\in F_0$. 

The harmonic map heat flow \eqref{heat1} with this initial map $u_0$ then has a long-time solution and converges to a harmonic map $u_\infty$ by Lemma \ref{lem:conv}.

By \eqref{db} in Corollary \ref{cor:asymp} and the distance bound observed above, we then conclude that $\lvert d(h(x), p_0) - d(u_\infty(x), p_0)\rvert$ is uniformly bounded for $x\in F_0$. Lemma \ref{lem:42}  then implies that the principal parts of $\text{Hopf}(h)$ and $\text{Hopf}(u_\infty)$ are identical at any pole of order at least three. For poles of order at most two, which we recall from Corollary \ref{cor:asymp} correspond to the cusps and boundary components of $X$, we apply \cite[Proposition 5.5]{Sagman} to conclude that the principal parts (i.e. leading coefficients) of  $\text{Hopf}(h)$ and $\text{Hopf}(u_\infty)$ at such punctures are identical.

This completes the proof of the existence of a harmonic map $u:=u_\infty$ with the desired asymptotics and principal parts. We now prove the second statement of Theorem \ref{thm:main}, concerning uniqueness. 

For this, if we assume that there are two such harmonic maps $u_1,u_2$ asymptotic to the same framing $\beta$, with the same principal parts, then we can show that the subharmonic  distance function $\Psi (x) = d(u_1(x), u_2(x))$ descends to a bounded function $\psi$ on $\hat{X}$ as follows:

It suffices to show that the distance function $\psi$ is bounded on neighborhoods of each pole, as their complement would be compact. For a neighborhood $U$ of a pole of order at most two, the boundedness of $\psi\vert_U$ is a consequence of  \cite[Lemma 4.8]{allegretti2021stability} and \cite[Lemma 5.9]{Sagman}. On  a neighborhood $U$ of a higher-order pole, the function $\psi$ is bounded by the following argument:  if the points $h(x), u(x)$ for $x\in \widetilde{U}$ both lie in a neighborhood of the same ideal point of the same chain, then from the geometry of such a cusp, $\Psi\vert_{\widetilde{U}} = d(h(x), u(x)) \approx  \lvert d(h(x), p_0) - d(u(x), p_0) \rvert$, that is, the two quantities are equal up to an additive error independent of $x$. By Lemma \ref{lem:42}, the latter distance function is bounded, and thus so is $\Psi\vert_{\widetilde{U}}$ and the quotient function $\psi\vert_U$.

Since $\psi$ is a bounded subharmonic function on a punctured Riemann surface (which is parabolic in the potential theoretic sense), it is constant, and in fact identically zero, i.e. $u_1=u_2$, by the same argument as in the proof of Lemma \ref{uniqueness}. This completes the proof.  \end{proof}

\appendix

\section{Interpolation of conformal metrics} 

In this appendix, we provide some details of the modification of conformal metrics that was alluded to while defining the metric $g$ on the domain in \S3.1 (see Definition \ref{metg}).

These modifications were 
\begin{itemize}
    \item[(a)]in a punctured-disk which is a neighborhood of a pole of order-$1$ of a meromorphic quadratic differential $q$, and 
    \item[(b)] a disk neighborhood of a  zero of order $n\geq 1$ of $q$,
\end{itemize} 
and we handle each in turn, recalling the defining constraints of the inperpolation. 

\subsection{Interpolation near a simple pole}
Choose conformal coordinates such that the neighborhood is $\mathbb{D}^\ast$ and the quadratic differential $q =  \frac{1}{4z}dz^2$. 

\vspace{.05in}

Recalling  the three constraints (a), (c) and (d) of the interpolation at the end of \S3.1.1, it suffices to find, for a choice of a sufficiently small  $\epsilon>0$ (to be determined later), a $C^2$-smooth metric $h$ on $\mathbb{D}^\ast$ (equipped with polar coordinates $(r,\theta)$) having  non-positive Gaussian curvature,  that agrees with the metric of a hyperbolic cusp for $0<r<\epsilon$, and with the flat metric induced by $4q$ in the annulus $2/3 <r <1$.  

\vspace{.05in}

Here, the hyperbolic cusp metric on $\mathbb{D}^\ast$ (where the cusp is at the puncture) is a conformal metric having  conformal factor $F(r) = \frac{1}{r^2\ln^2(r)}$, while the flat $q$-metric has conformal factor $G(r) = \frac{1}{r}$. Since these are rotationally symmetric, we shall assume that the metric $h$ has the form $h(r)|dz|^2$, where  the conformal factor $h: (0,1]\rightarrow (0,\infty)$ is a $C^2$-smooth interpolating function that agrees with $F$ on $(0, \epsilon]$ and with $G$ on $[2/3, 1)$. (See Figure \ref{graphs}.)

\begin{figure}
\begin{center} 
\includegraphics[scale=.6]{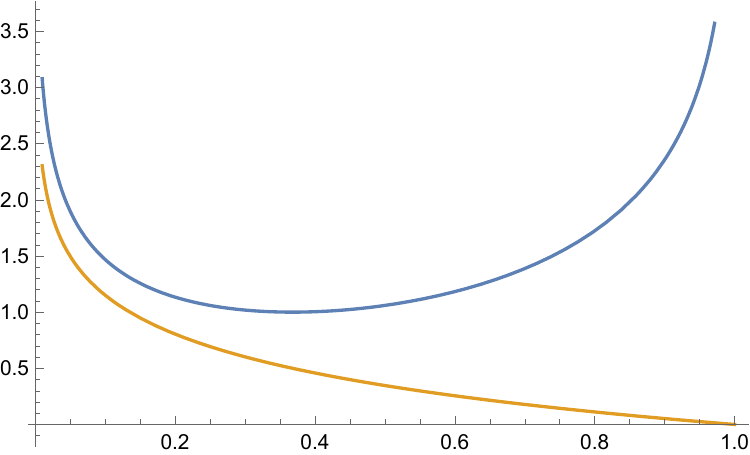}
\caption{The graphs of the functions $f$ (in blue) and $g$ (in yellow) on the interval $(0,1)$. We seek to find a $C^2$-smooth function $u$ that interpolates between $f$ and $g$ on $[\epsilon, 2/3)$. }\label{graphs}
\end{center}
\end{figure} 

\vspace{.05in}

Expressing the conformal metric as $e^{2u(r)}|dz|^2$, where $u(r) := \frac{1}{2} \ln h(r)$, this is equivalent to finding a $C^2$-smooth function $u:(0,1) \to \mathbb{R}^+$ that agrees with $f(r) := \frac{1}{2} \ln F(r)$ on $(0, \epsilon]$ and with $g(r) := \frac{1}{2}\ln G(r)$ on $[2/3, 1)$.  Note that it suffices to find $u:[\epsilon,2/3] \to \mathbb{R}^+$ such that the values and derivatives of $u$ up to second order agree with those of $f$ at $r=\epsilon$, and with those of $g$ at $r= 2/3$. 

\vspace{.05in}

Recall that for the conformal metric $e^{2u(r)}|dz|^2$, the Gaussian curvature is given by
\begin{equation}\label{eq:K}
K(r)= - e^{-2u}\Delta u= -\frac{u''(r)+\frac{1}{r}u'(r)}{e^{2u(r)}}.
\end{equation}
Thus the non-positive curvature condition is satisfied if   $u''(r)+\frac{1}{r}u'(r)\geq 0$ for $r\in (0,1)$.  Since $r$ is positive, this condition is equivalent to checking that 
\begin{equation*}
    k(r) := ru^\prime(r)
\end{equation*} is a monotonically increasing function.  Note that this holds when $u(r) = f(r)$ (when the curvature $K(r) \equiv -1$) and for $u(r) = g(r)$ (when the curvature $K(r) \equiv 0$).

\vspace{.05in}

We shall require that on $(0,\epsilon]$ we have $k(r) = r f^\prime (r) = -1 - \frac{1}{\ln{r}} $, while on $[2/3,1)$ we have $k(r) = rg'(r) = -\frac{1}{2}$ (the constant function). Note that for any sufficiently small $\epsilon>0$ we have the strict inequality $k(\epsilon) = \epsilon f'(\epsilon)< \frac{2}{3}g'(\frac{2}{3}) = k(\frac{2}{3})$, so a monotonically increasing interpolation on $(\epsilon, 2/3)$ is possible. 

\vspace{.05in}

If we choose a non-negative continuous function $v:[\epsilon, 2/3] \to \mathbb{R}_{\geq 0} $ such that 
\begin{itemize}
    \item[(i)] $\int\limits_\epsilon^{2/3} v(s) ds  = k(\frac{2}{3}) - k(\epsilon) >0$,
    \item[(ii)] $v(\epsilon) = k'(\epsilon) = (\epsilon \ln^2\epsilon)^{-1} >0 $, and  \item[(iii)] $v(\frac{2}{3}) = k'(\frac{2}{3})=0$
\end{itemize}
then defining 
\begin{equation}\label{keq}
k(r) := k(\epsilon) + \int\limits_\epsilon^r v(s) ds
\end{equation}
for $r\in [\epsilon, 2/3]$  provide the desired interpolation which is $C^1$-smooth on $(0,1)$.

Such a function  $v$ is easy to find: indeed, we can choose non-negative continuous functions satisfying (ii) and (iii) with the integral $\int\limits_\epsilon^{2/3} v(s) ds  $ taking \textit{any} prescribed value, in particular the one desired in (i). See Figure \ref{func-v} for one way of constructing such a function $v$; in particular, we can make a choice of $v$ such that its support is contained in two small intervals, one of the form $(\epsilon, \epsilon_0)$ and the other of the form $(a,b)$ where $\epsilon_0<a<b<\frac{2}{3}$. As we shall now explain, we also need a \textit{fourth} condition on $v$, for which the above flexibility in choosing $v$ will be useful.
\vspace{.05in}

\begin{figure}[h]
\begin{center} 
\includegraphics[scale=.2]{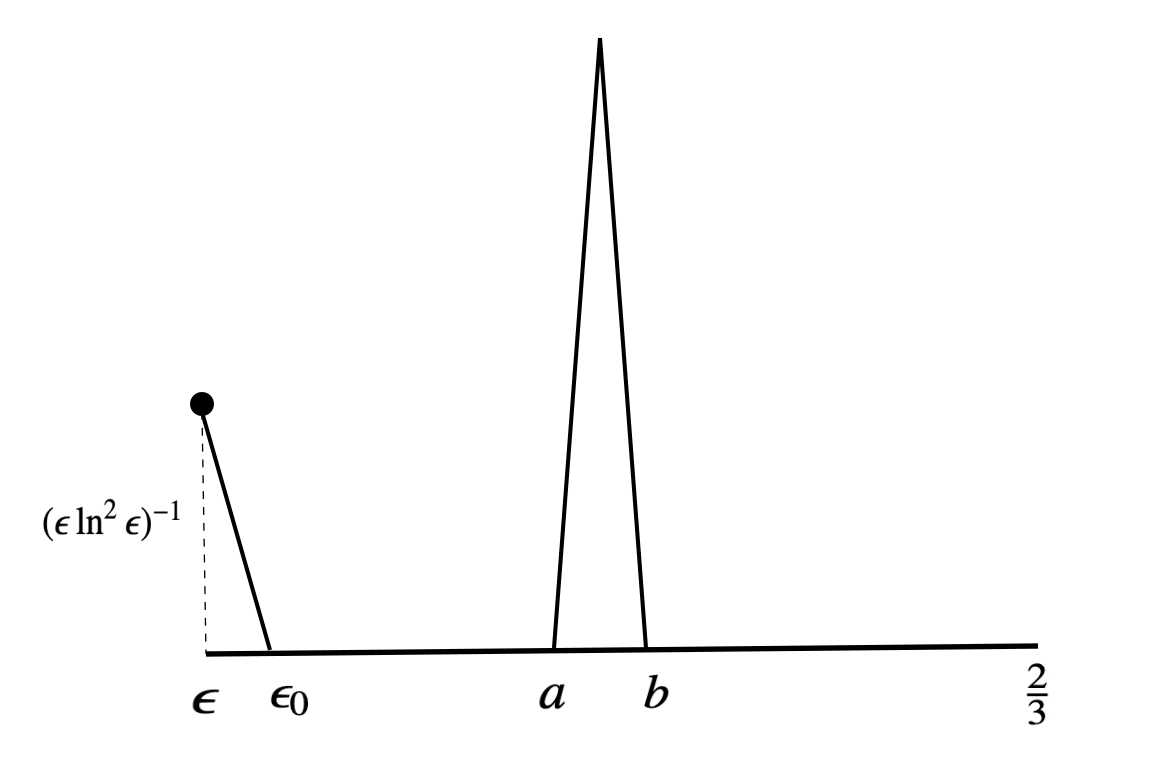}
\caption{The function $v$ on $(\epsilon,2/3)$ can be chosen to have a graph as shown. Here $\epsilon_0,a,b$ can be chosen such that $\epsilon_0 - \epsilon$ and $b-a$ are arbitrarily small.}\label{func-v}
\end{center}
\end{figure}

Recall that our goal is to construct an interpolation $u:(0,1) \to \mathbb{R}^+$. Since $k(r) = r u'(r)$, we can define
\begin{equation}\label{ueq}
u(r) = f(\epsilon) + \int\limits_\epsilon^{r} \frac{k(s)}{s} ds
\end{equation} 
for $r\in [\epsilon, 2/3]$.
From the properties satisfied by $k$, this would automatically satisfy $u(\epsilon) =f(\epsilon)$, $u'(\epsilon) = f'(\epsilon)$, $u'(\frac{2}{3}) = g'(\frac{2}{3})$, $u''(\epsilon) = f''(\epsilon)$ and $u''(\frac{2}{3}) = g''(\frac{2}{3})$. The only additional condition that we need to satisfy is $u(\frac{2}{3}) = g(\frac{2}{3})$, for which we would need $\int\limits_\epsilon^{2/3} \frac{k(s)}{s} ds = g(\frac{2}{3}) - f(\epsilon)$.

Plugging back in \eqref{keq}, we obtain:
$\int\limits_\epsilon^{2/3} \frac{k(s)}{s} ds = k(\epsilon) \ln\frac{2}{3\epsilon} + \int\limits_\epsilon^{2/3} (\int\limits_\epsilon^r v(s)ds)r^{-1} dr$ and hence the fourth condition that $v$ needs to satisfy is:
\begin{itemize}
    \item[(iv)] $\mathcal{I} := \int\limits_\epsilon^{2/3} (\int\limits_\epsilon^r v(s)ds)r^{-1} dr = g(\frac{2}{3}) - f(\epsilon) - k(\epsilon) \ln\frac{2}{3\epsilon} $ 
\end{itemize}
where note that the RHS of the above expression is positive for sufficiently small $\epsilon>0$, since $f(\epsilon) = -\ln{(\epsilon\lvert \ln\epsilon\rvert)}$ and $k(\epsilon) = -1 - \frac{1}{\ln\epsilon} \approx -1$  and $\ln\frac{2}{3\epsilon} \gg  \ln{(\frac{1}{\epsilon\lvert \ln\epsilon\rvert})}$ for sufficiently small $\epsilon$.

\vspace{.05in}

We claim that in fact, we can choose $v$ such that the integral $\mathcal{I}$ on the LHS of (iv) is \textit{any} positive value that is less than $\alpha := (k(\frac{2}{3}) - k(\epsilon))\ln\frac{2}{3\epsilon}$. Here, note that $\alpha$ is an upper bound for $\mathcal{I}$ from condition (i) on $v$, and it is easy to see that the RHS of (iv) is less than $\alpha$.

To see, this, note that for a choice of $v$ with support in $(\epsilon,\epsilon_0) \cup (a,b)$ as in Figure  \ref{func-v}, we have:
\begin{equation}\label{ieq}
    \mathcal{I} \approx \int\limits_{b}^{2/3} \left(k(2/3) - k(\epsilon)\right) r^{-1} dr = \left(k(2/3) - k(\epsilon)\right) \ln \frac{2}{3b}
\end{equation}
where ``$\approx$" denotes an additive error that can be made arbitrarily small by making  the differences $\epsilon_0-\epsilon$ and $b-a$ small. 
By varying the endpoint $b$ in $(\epsilon, 2/3)$, it follows from \eqref{ieq} that $\mathcal{I}$ can be chosen to be any value in the range $(0,\alpha)$.
Hence, requirement (iv) can be satisfied and with this choice of $v$, equations \eqref{keq} and \eqref{ueq} provide the desired interpolating function $u$.

\subsection{Interpolation near a zero}
Choose conformal coordinates such that the neighborhood is $\mathbb{D}$ and the quadratic differential $q = \frac{1}{4} z^ndz^2$ where $n\geq 1$ is the order of the zero of $q$.

\vspace{.05in}

Recalling  the three constraints (b), (c) and (d) of the interpolation at the end of \S3.1.1, we need to find a $C^2$-smooth metric $h$ on $\mathbb{D}$ (equipped with polar coordinates $(r,\theta)$) having  non-positive Gaussian curvature, that agrees with the flat $4q$-metric in a collar neighborhood of $\partial \mathbb{D}$, say in the annulus $A_\epsilon = \{\epsilon< r<1\}$ for some positive $\epsilon$. 

\vspace{.05in}

Consider the conformal metric $f_{\epsilon}(z)|dz|^2$ where the conformal factor $f_{\epsilon}$ is given by
\[f_{\epsilon}(z)=\begin{dcases}
\psi(r) &0<r<\epsilon\\
r^{n}& z\in A_{\epsilon} \\
\end{dcases}\]  
where $\psi(r)=ar^{4{n}}+br^{2{n}}+c$ is a polynomial with coefficients $a,b,c$ that we shall now prescribe. 

\vspace{.05in} 

To  ensure that  $f_{\epsilon}$ is a $C^2$-smooth function, we need to choose $a,b,c$ such that  the derivatives up to order $2$ at $r = \epsilon$ match. To do this, we solve the following linear system:
\begin{align}\label{soe}
\begin{cases}
&\psi(\epsilon)=a\epsilon^{4{n}}+b\epsilon^{2{n}}+c=\epsilon^n,\\
&\psi'(\epsilon)= 4{n}a\epsilon^{4{n}-1}+2nb\epsilon^{2{n}-1}= n\epsilon^{n-1}\\
&\psi''(\epsilon)=4{n}(4{n}-1)a\epsilon^{4n-2}+2{n}(2{n}-1)b\epsilon^{2{n}-2}= n(n-1)\epsilon^{n-2}.
\end{cases}
\end{align}
to obtain:
\[a=-\frac{1}{8n\epsilon^{3n}},b=\frac{3}{4\epsilon^n}, \quad \text{and}\quad c=\left(1+\frac{1}{8{n}}-\frac{3}{4}\right)\epsilon^n.\]
 Note that $a<0$ and $b,c>0$.

 \vspace{.05in}

For the non-positivity of the Gaussian curvature, from \eqref{eq:K} it suffices to show that $\Delta \ln \psi \geq 0$. Note that \[\psi'(r)=4nar^{4n-1}+2nbr^{2n-1}, \ \psi''(r)=4n(4n-1)ar^{4n-2}+2n(2n-1)br^{2n-2},\] which in turn yields,
\begin{align*}
    &\Delta \ln \psi=-\frac{\psi'^2}{\psi^2}+\frac{\psi''}{\psi}+\frac{\psi'}{r\psi} \\
    =&\frac{[4n(4n-1))ar^{4n-2}+2n(2n-1)br^{2n-2}+4anr^{4n-1}+2bnr^{2n-1}]\psi(r)-(4nar^{4n-1}+2nbr^{2n-1})^2}{\psi^2}\\
    =&\frac{4n^2r^{2n-2}(abr^{4n}+4acr^{2n}+bc)}{\psi^2}.
\end{align*}

Thus, it suffices to prove that $p(r) = (abr^{4n}+4acr^{2n}+bc)\geq 0$.  From basic calculus it is easy to check that $p(r)$ attains its minimum at a point $r=\alpha$ where $\alpha^{2n}=-\frac{4ac}{2ab}$. Therefore \[p(r)\geq p(\alpha)=ab\left(\frac{-4ac}{2ab}\right)^2-4ac\frac{4ac}{2ab}+bc=\frac{-(4ac)^2}{4ab}+bc\geq0,\] since recall that $a<0$ and $b,c>0$. This completes the proof.

\bibliographystyle{amsalpha}
\bibliography{thesis_ref}

\end{document}